\documentclass[11pt,reqno]{amsart}

\usepackage{srcltx}
\usepackage{bm}
\usepackage{amsmath,amssymb,cases,color,mathtools}
\usepackage{hyperref}
\usepackage[capitalise]{cleveref}
\usepackage{vmargin}
\setmarginsrb{1.5cm}{1cm}{1.5cm}{3cm}{1cm}{1cm}{2cm}{2cm}
\usepackage{bigints}
\usepackage{relsize}
\usepackage{stmaryrd}

\usepackage{amsmath, amssymb,amscd, }
\usepackage{amsfonts}
\usepackage{mathrsfs}
\usepackage{graphicx}
\usepackage{xcolor}

 \usepackage{upref}
\hypersetup{linkcolor=blue, colorlinks=true,citecolor = red}
\newtheorem{theorem}{Theorem}[section]
\newtheorem{proposition}[theorem]{Proposition}
\newtheorem{lemma}[theorem]{Lemma}
\newtheorem{cor}[theorem]{Corollary}
\newtheorem{remark}[theorem]{Remark}
\newtheorem{defin}[theorem]{Definition}

\theoremstyle{definition}

\theoremstyle{remark}
\numberwithin{equation}{section}
\newcommand{\Q}{Q^{n}}
\newcommand{\uu}{u^{n}}  
\newcommand{\mcF}{\mathcal{F}}
\newcommand{\mcS}{\mathcal{S}}

\newcommand{\na}{\nabla}
\newcommand{\pa}{\partial}
\newcommand{\f}{\frac}
\newcommand{\G}{\Gamma}
\newcommand{\de}{\delta}
\newcommand{\Om}{\Omega}
\newcommand{\lam}{\lambda}
\newcommand{\ri}{\rightarrow}

\begin{document}

\title[NLC-Colloid]{Global existence of Weak Solutions  for a model of\\  nematic liquid crystal-colloidal interactions}




\author{Zhiyuan Geng $^1$
		\and Arnab Roy$^4$ \and Arghir Zarnescu$^{1,2,3}$}
  \thanks{Z.G, A.R and A.Z have been partially supported by the Basque Government through the BERC 2022-2025 program and by the Spanish State Research Agency through Severo Ochoa CEX2021-001142 and through project PID2020-114189RB-I00 funded by Agencia Estatal de Investigación (PID2020-114189RB-I00 / AEI / 10.13039/501100011033). A.Z. was also partially supported  by a grant of the Ministry of Research, Innovation and Digitization, CNCS - UEFISCDI, project number PN-III-P4-PCE-2021-0921, within PNCDI III. The research of A.R. has been supported by the Alexander von Humboldt-Stiftung / Foundation. The authors acknowledge the hospitality in Hausdorff Research Institute for Mathematics funded by the Deutsche Forschungsgemeinschaft (DFG, German Research Foundation) under Germany's Excellence Strategy–EXC-2047/1–390685813.}

\date{\today}

\maketitle

\bigskip

\centerline{$^1$ BCAM, Basque Center for Applied Mathematics}

\centerline{Mazarredo 14, E48009 Bilbao, Bizkaia, Spain}

\centerline{$^2$IKERBASQUE, Basque Foundation for Science, }

\centerline{Plaza Euskadi 5, 48009 Bilbao, Bizkaia, Spain}

\centerline{$^3$``Simion Stoilow" Institute of the Romanian Academy,}

\centerline{21 Calea Grivi\c{t}ei, 010702 Bucharest, Romania }

\centerline{$^4$ Technische Universit\"{a}t Darmstadt}

\centerline{Schlo\ss{}gartenstra{\ss}e 7, 64289 Darmstadt, Germany}
\begin{abstract}
In this paper we study a mathematical model describing the movement of a colloidal particle in a fixed, bounded three dimensional container filled with a nematic liquid crystal fluid. The motion of the fluid is governed by the Beris–Edwards model for nematohydrodynamics equations, which  couples the incompressible Navier-Stokes equations with a  parabolic system. The dynamics of colloidal particle within the nematic liquid crystal is described by the conservation laws of linear and angular momentum.  We prove the existence of  global  weak solutions for the coupled system.
\end{abstract}



\bigskip

\section{Introduction}

 We consider a colloidal particle as a rigid body  which is closed, bounded, simply connected and denote its position at time $t$  by $\mathcal{S}(t) \subset \mathbb{R}^3$, for any $t \in (0,T)$. The colloidal particle is moving inside a bounded $C^2$ domain $\Omega \subset \mathbb{R}^{3}$. 
 We assume that the fluid domain $\mathcal{F}(t):=\Omega \setminus \mathcal{S}(t)$ is filled with a viscous incompressible fluid with nematic liquid crystals. 
The initial domain of the colloidal particle is denoted by $\mathcal{S}_0$ and is assumed to have a $C^2$ boundary. Correspondingly, $\mathcal{F}_0 =\Omega \setminus \mathcal{S}_0$ is  the initial fluid domain. The solid domain at time $t$ is given by
\begin{equation*} \mathcal{S}(t)= \left\{h(t)+ \mathbb{O}(t)y \mid y\in \mathcal{S}_{0}\right\}, \end{equation*}
where $h(t)$ is the centre of mass of the body and $\mathbb{O}(t) \in SO(3)$ is associated to the rotation of the rigid body. The velocity of the rigid body is described by
\begin{equation*} 
u^{\mathcal{S}}(x,t)= \ell(t)+  \omega(t) \times (x-h(t)) \; \mbox{for} \quad (t,x) \in (0,T)\times \mathcal{S}(t), \end{equation*} 
where $\ell(t)=h'(t)$ and $\omega(t)$ are the linear and angular velocities of the centre of mass of the body. 
  
The  nematohydrodynamic equations coupling the equations describing the evolution of the direction of the molecules, given by $Q$, with those for the velocity of the center of mass of the  nematic molecules, $u^{\mcF}$, are:
\begin{align} {\partial_t Q} + (u^{\mcF}\cdot \nabla) Q - S(\nabla u^{\mcF}, Q) &= \Gamma H   \quad \mbox{ in } (0,T)\times \mathcal{F}(t), \label{Q:fluid} \\ {\partial_t u^{\mcF}} + \left(u^{\mcF}\cdot \nabla\right)u^{\mcF} + \nabla p^{\mcF} &=\mu\Delta u^{\mcF} +  \operatorname{div}(\tau +\sigma) \quad \mbox{ in } (0,T)\times \mathcal{F}(t),\label{vel:fluid}  \\ 
 \mathrm{div}\,u^{\mcF}&=0 \quad \mbox{ in } (0,T)\times\mathcal{F}(t),
 \label{div:fluid}
\end{align}
with 
\begin{equation*}
S(\nabla u^{\mcF}, Q)= (\xi D(u^{\mcF}) +\Sigma(u^{\mcF}))\left(Q+\frac{1}{3}\mathbb{I}_3\right) + \left(Q+\frac{1}{3}\mathbb{I}_3\right)(\xi D(u^{\mcF}) - \Sigma(u^{\mcF})) -2\xi\left(Q+\frac{1}{3}\mathbb{I}_3\right)\mathrm{tr}(Q\nabla u^{\mcF}),
\end{equation*}
\begin{equation}\label{def: H}
H=\Delta Q -a Q + b\left(Q^2 - \frac{\operatorname{tr}(Q^2)}{3}\mathbb{I}_3\right) - c\ Q\operatorname{tr}(Q^2).
\end{equation}
In the above the constants $c>0$, $\Gamma >0$, $\mu>0$, $a, b \in \mathbb{R}$. Here $u^{\mcF}$ is the flow velocity, $p^{\mcF}$ is the pressure, $Q$ is the Landau-de Gennes tensor order parameter that is a traceless and symmetric matrix and the Newtonian viscosity coefficient is denoted by $\mu$. The constant $\xi$ is particular to the liquid crystal material and measures the ratio between the tumbling and the aligning effect that a shear flow would exert over the liquid crystal directors.

The symmetric and anti-symmetric parts of Newtonian stress tensor are given by 
$$D(u^{\mcF}) =\frac{1}{2}\left(\nabla u^{\mcF} + (\nabla u^{\mcF}) ^{\top}\right),\quad \Sigma(u^{\mcF})= \frac{1}{2}\left(\nabla u^{\mcF} - (\nabla u^{\mcF}) ^{\top}\right).$$
The additional stress,  the non-Newtonian strain tensor, describing the effect produced by the rod-like liquid crystal molecules,   has the following symmetric and anti-symmetric parts: 
\begin{equation*}
\tau = \left[-\xi\left(Q+\frac{1}{3}\mathbb{I}_3\right)H-\xi H\left(Q+\frac{1}{3}\mathbb{I}_3\right) + 2\xi\left(Q+\frac{1}{3}\mathbb{I}_3\right)QH - \nabla Q\odot \nabla Q\right],
\end{equation*}
\begin{equation*}
\sigma = QH-HQ.
\end{equation*}
In this paper, we restrict ourselves to the co-rotational case $\xi = 0$ which reduces $\tau$ and $S(\nabla u^{\mcF}, Q)$ to the following form:
\begin{equation}\label{tau}
\tau = - \nabla Q\odot \nabla Q,
\end{equation}
and 
\begin{equation*}
S(\nabla u^{\mcF}, Q)= \Sigma(u^{\mcF}) Q- Q\Sigma(u^{\mcF}).
\end{equation*}
The quantity $H$ (defined in \eqref{def: H}) also relates to the variational derivative of the free energy functional:
\begin{equation*}
\mathcal{E}(Q)=\int_{\mcF(t)}\left(\frac{1}{2}|\nabla Q|^2+ f_{b}(Q)\right),
\end{equation*}
where the bulk energy $f_b(Q)$ is given by 
\begin{equation*}
f_b(Q)=\frac{a}{2}\textrm{tr}(Q^2)-\frac{b}{3}\textrm{tr}(Q^3)+\frac{c}{4}(\textrm{tr}(Q^2))^2,
\end{equation*}
where $a,b,c$ are material dependent constants. Thus $H$ can be written as
\begin{equation}\label{newH}
H=\Delta Q - \frac{\partial f_b(Q)}{\partial Q}.
\end{equation}
We take the colloidal particle to have constant density equal to $1$. Then the  total mass and the position of the centre of mass of the colloid are given by 
\begin{equation}\label{cm}
m = \int\limits_{\mathcal{S}(t)} dx,\quad h(t)= \frac{1}{m}\int\limits_{\mathcal{S}(t)} x \ dx.
\end{equation}
Further, 
the moment of inertia $J(t)$ is defined by
\begin{equation}\label{moment of inertia}
J(t) =  \int\limits_{\mathcal{S}(t)}
 \Big(|x-h(t)|^2\mathbb{I}_3 - (x-h(t))\otimes (x-h(t))\Big)\, dx. 
 \end{equation}
The motion of the colloidal particle can be described by Newton's law:
\begin{align} 
m\ell'&= -\int\limits_{\partial \mathcal{S}(t)} (2\mu D(u^{\mcF})+\tau+\sigma-p^{\mcF}\mathbb{I}_3) n\, d\Gamma , \label{linear momentum:body}
\\ 
J\omega' &= -\int\limits_{\partial \mathcal{S}(t)} (x-h) \times (2\mu D(u^{\mcF})+\tau+\sigma-p^{\mcF}\mathbb{I}_3) n\, d\Gamma,
\label{angular momentum:body}
\end{align} 
where $n$ is the unit outward normal to the boundary of $\mathcal{F}(t)$, i.e., directs towards $\mathcal{S}(t)$. In the system \eqref{Q:fluid}--\eqref{angular momentum:body}, the unknowns are $u^{\mcF}(t,x)$, $Q(t,x)$, $h(t)$, $\ell(t)$ and $\omega(t)$. We consider the following  boundary condition at the fluid-structure interface for the flow velocity and the homogeneous Dirichlet boundary condition for the tensor order parameter: 
\begin{align} 
u^{\mcF} &= u^{\mathcal{S}} \quad \mbox{ on }  (0,T)\times\partial\mathcal{S}(t),\quad u^{\mcF}=0\quad \mbox{ on }  (0,T)\times\partial\Omega,\label{bdary:u} \\
 { Q} &= 0  \quad \mbox{ on } (0,T)\times(\partial\mathcal{S}(t)\cup \partial\Omega). \label{bdary:Q}
\end{align}

We can close the system by assigning the appropriate initial conditions:
\begin{equation}\label{initial}
Q(0,x)=Q_0(x) \mbox{ in } \mathcal{F}_0, \quad u(0,x)=u_0(x) \mbox{ in } \mathcal{F}_0,\quad h(0)=0, \quad h'(0)=\ell_0\in \mathbb{R}^3,\quad \omega(0)=\omega_0\in \mathbb{R}^3.
\end{equation}


  There is a rich literature on nematohydrodynamics in the absence of particles, see for instance the review \cite{lin2014recent} and the references therein. Global existence of the weak solution for the full coupled Navier-Stokes and Q-tensor system was established by Paicu-Zarnescu in the whole space \cite{MR2837493, MR2864407} and by Abels-Dolzmann-Liu in a bounded domain \cite{abels2014well} with further results being obtained in \cite{abels2016strong,wilkinson2015strictly,guillen2014weak,du2020suitable,cavaterra2016global,huang2015global,dai2016asymptotic}.  On the other hand, not many results are available regarding suspended colloidal particles in a nematic-fluid flow except the experimental results \cite{SEN}, \cite{MU}. In \cite{MO}, some interesting features in a two-dimensional NLC (Nematic liquid crystal) microfluidic channel (parallel-plate geometry) with a circular particle has been observed experimentally and numerically. 
  
The mathematical analysis of systems describing the motion of a rigid body in a viscous \textit{incompressible} fluid is nowadays well developed.  The proof of existence of weak solutions until a first collision can be found in several papers, see \cite{CST,DEES1,GLSE,HOST,SER3}. Later, the possibility of collisions was included in the works of San Martin et.al in \cite{MR1870954} and of Feireisl in \cite{F3} for $2D$ and $3D$ cases respectively.
In the case of nematic fluid flow, there is the additional stress tensor and the Navier-Stokes equations are coupled with the equation describing the orientation of the anisotropic liquid crystal molecules. Moreover, there is a competition between elasticity and the interaction between NLC molecules and surfaces (known as ``anchoring'') in the nematic fluid-colloid interaction. This makes this interaction problem qualitatively different from the incompressible Newtonian fluid-rigid body interaction.

However there is, to our knowledge, no mathematical analysis available in terms of the appropriate sense of solutions (weak or strong) or regarding the pertinent functional framework to analyze the nematic fluid-rigid body interaction problems. In this paper, our aim is to establish existence of weak solution to system \eqref{Q:fluid}--\eqref{initial} in any interval $[0,T]$. Our work builds on the approaches in the Newtonian case, in particular the works  \cite{MR1870954} and \cite{F3} but deals with certain specific challenging difficulties related to the presence of the non-Newtonian component, the liquid crystal molecules: the more complex system has an additional equation when compared to the Newtonian case, namely the equation for $Q$. This equation is defined only on the fluid domain $\mcF(t)$ and, unlike in the Newtonian case, there is no obvious way of extending the equation to the whole container $\Omega$, particularly because the natural level of regularity of the weak solutions, that is one derivative higher than that for the fluid component. We combine ideas specific to the treatment of liquid crystal by introducing an auxiliary equation for the variational derivative of the free energy, as in  \cite{MR3274285}  together with a specific penalisation for enforcing the $Q$ to be  extended by zero throughout the colloid, while taking care of the complex effects on the fluid with an additional temporary regularisation inspired by \cite{HOST}.

 We point out that for the sake of convenience, our proof will focus on the flow with only one colloidal particle, but all of our results can be extended without much difficulty to the multi-particle case (see Remark \ref{rmk:multi-particle}). The plan of the paper is as follows: In Section \ref{sec2}, we set the notations, define a notion of weak solution to the system \eqref{Q:fluid}--\eqref{initial}, and state the main result (see Theorem \ref{exist:main}) of the paper. Section \ref{sec3} is devoted to providing an appropriate approximate problem and to proving the existence of a solution to this problem. Section \ref{sec4} is dedicated to proving the main result, Theorem \ref{exist:main} by passing to the limit in the penalized parameters. In the \cref{appendix}, we explain the derivation of the weak formulation \eqref{ws:u}.

\section{Weak solutions: notations and the main result}\label{sec2}
We need to introduce some notations before defining a \textbf{weak solution} to system \eqref{Q:fluid}--\eqref{initial}. Firstly, we want to define appropriate space for the fluid velocity. Let us introduce
\begin{equation*}
\mathcal{D}(\Omega)=\left\{v\in C_c^{\infty}(\Omega;\mathbb{R}^d)\mid \operatorname{div}v=0\mbox{ in }\Omega\right\}.
\end{equation*}
The spaces $H(\Omega)$ and $V(\Omega)$ are the closure of the space $\mathcal{D}(\Omega)$ under the norms $L^2(\Omega)$ and $H^1(\Omega)$ respectively. We can also characterize these spaces in the following way:
\begin{equation*}
{H}(\Omega)=\left\{v\in L^{2}(\Omega;\mathbb{R}^d)\mid \operatorname{div}v=0\mbox{ in } [C_c^{\infty}(\Omega)]',\quad v\cdot n=0\mbox{ in }H^{-1/2}(\partial\Omega)\right\},
\end{equation*}
\begin{equation*}
{V}(\Omega)=\left\{v\in H_0^{1}(\Omega;\mathbb{R}^d)\mid \operatorname{div}v=0\mbox{ in }L^2(\Omega)\right\}.
\end{equation*}
If $\chi$ is the characteristic function of a subset in $\Omega$, we define 
\begin{equation*}
    \mcS(\chi)=\{x\in\Om: \chi(x)=1\}.
\end{equation*}

Let $\sigma>0$ and $G\subset \Om$ be an open subset. We denote by $G_\sigma$ the $\sigma$-neighborhood of $G$, i.e.
\begin{equation*}
    G_\sigma:=\{x\in\Om: \mathrm{dist}(x,G)<\sigma\}.
\end{equation*}

We define the function space
\begin{equation*}
{K}_\sigma(\chi):=\left\{v\in V(\Omega)\mid D(v)(x)=0\mbox{ for a.e. } x\in \mcS(\chi)_\sigma\right\},
\end{equation*}
and 
\begin{equation*}
\begin{split}
    &K_0(\chi):=\mbox{ closure of }\bigcup\limits_{\sigma>0}K_\sigma(\chi)\mbox{ in }V(\Omega).\\
    &K(\chi):=\{v\in V(\Omega)\mid \chi D(v)=0\text{ in }L^2(\Omega)\}.
\end{split}
\end{equation*}

A function $\varphi\in \mbox{Char}([0,T]\times\Omega)$ iff $\varphi$ is the characteristic function of some subset of $[0,T]\times \Omega$. If $\varphi\in \mbox{Char}([0,T]\times \Omega)$, then
\begin{equation*}
L^p(0,T;K(\varphi))=\left\{v\in L^p(0,T; V(\Omega))\mid v(\cdot,t) \in K(\varphi(\cdot,t))\mbox{ for almost all }t\in [0,T]\right\}.
\end{equation*}
Secondly, we want to define the appropriate space for Landau-de Gennes orientation $Q$ tensor. We introduce the space of $Q$-tensors, that is of $3\times 3$ symmetric, traceless, real-valued matrices:
\begin{equation*}
M=\left\{Q\in \mathbb{R}^{3\times 3}; Q=Q^{\top};\mid \mbox{tr}(Q)=0\right\}.
\end{equation*}
Instead of considering separately the velocity field of the fluid and the rigid body, we consider here one uniform velocity field $u$ defined in $(0,T)\times \Omega$ whose restriction to the rigid body $\mathcal{S}(t)$ is the rigid velocity $u^{\mathcal{S}}$, i.e,
\begin{equation}\label{def:u}
  u = \begin{cases}
  u^{\mcF} \mbox{ in }\mcF(t), \\
  u^{\mcS} \mbox{ in }\mcS(t).
  \end{cases}
  \end{equation}
  
Throughout the paper, for two vectors $a,b\in\mathbb{R}^3$ we set $a\cdot b=\sum_{i=1}^3 a_ib_i$ and $a\otimes b= ab^{\top}=(a_ib_j)_{1\leq i,j\leq 3}$; while for two matrices $A,B\in \mathbb{R}^{3\times 3}$ we set 
\begin{equation*}
A:B=\sum_{i,j=1}^3 A_{ij}B_{ij},\quad \na A \odot \na B=(\pa_i A:\pa_j B)_{1\leq i,j\leq 3} \mbox{  and  }Q:\na A = (Q:\pa_j A)_{1\leq j\leq 3}. 
\end{equation*}

\begin{defin}\label{def:weak sol}
Let $u_0 \in H(\Omega)$, $Q_0 \in H^1_0(\mathcal{F}_0;M)$ and $\varphi_0$ be the characteristic function of $\mcS_0$.  The quadruple $(u,Q,H,\varphi)$ satisfying
\begin{equation*}
\begin{split}
&u \in L^{\infty}(0,T; H(\Omega))\cap L^2(0,T;K(\varphi)),\quad Q\in L^{\infty}(0,T; H^1_0(\mcF(t);M))\cap L^2(0,T;H^2(\mcF(t);M)),\\
&\varphi\in \mbox{Char}\ (\Omega)\cap C^{0,\frac{1}{p}}(0,T;L^p(\Omega)),\ p\in[1,+\infty), \quad H\in L^2(0,T,L^2(\mathcal{F}(t)))
\end{split}
\end{equation*}
is said to be a weak solution of the system \eqref{Q:fluid}--\eqref{initial} if the following holds
\begin{enumerate}
\item For any functions 
 \begin{align*}
 &\psi\in H^1((0,T); H^1(\Omega;M) )\cap L^2(0,T;H_0^1(\mcF(t);M)),\  \psi(T)=0,\\
 &\xi\in L^2(0,T;H^1_0(\mcF(t);M)), \\
 &\zeta\in H^1((0,T)\times \Omega))\cap L^2(0,T;K_0(\varphi)), \ \zeta(T)=0,\\
 &\eta\in C^1([0,T]\times\Omega), \ \eta(T)=0.   
 \end{align*}
\begin{equation}\label{ws:Q}
\int\limits_0^T\int\limits_{\mcF(t)} \left[- Q:\pa_t \psi -\Gamma H : \psi +(u\cdot \nabla)Q : \psi -(\Sigma Q) : \psi + (Q\Sigma): \psi \right]\ dx \ dt =\int\limits_{\mcF(0)} Q(0):\psi(0)\,dx
\end{equation}
\begin{equation}\label{ws:H}
\int_0^T\int\limits_{\mcF(t)} \left[ \na Q:\na \xi+ \frac{\pa f_b(Q)}{\pa Q}: \xi+H : \xi \right]\ dx\ dt=0
\end{equation}
\begin{multline}\label{ws:u}
\int\limits_0^T\int\limits_{\Omega} \left[-u\cdot \pa_t\zeta - (u\otimes u): \na\zeta +\mu D(u):D(\zeta)\right]\ dx \ dt \\
= \int_{\Omega} u(0)\cdot \zeta(0)\,dx+ \int\limits_0^T\int\limits_{\mcF(t)} \left[(\na Q\odot \na Q):\na \zeta -(QH):\na \zeta+ (HQ):\na\zeta \right]\ dx \ dt,
\end{multline}
\begin{equation}\label{ws:solid}
\int\limits_0^T\int\limits_{\Omega} \varphi \left[\partial_t\eta+(u\cdot\nabla)\eta\right]\ dx\ dt= - \int\limits_{\Omega} \varphi_0 \eta(0)\ dx,
\end{equation}
\item For almost every $t\in (0,T)$, the following energy inequality holds: 
\begin{equation}\label{energy:main}
E(u,Q)(t)+\mu\int_0^t \int\limits_{\Omega}|D u|^2\,dx\ ds 
   + \Gamma\int_0^t\int\limits_{\mathcal{F}(t)} |H|^2\,dx\,ds
\le E(u_0,Q_0),\ \ \forall \ t\in[0,T],
\end{equation}
where the total energy $E(u,Q)$ is defined by
\begin{equation}\label{energy}
E(u,Q)\stackrel{def}{=}\frac{1}{2}\int\limits_{\Omega}|u|^2\,dx+\mathcal{E}(Q),
\end{equation}
\begin{equation*}
\mathcal{E}(Q)=\int\limits_{\mcF(t)}\left(\frac{1}{2}|\nabla Q|^2+ f_{b}(Q)\right)\,dx\quad\mbox{with}\quad f_b(Q)=\frac{a}{2}\textrm{tr}(Q^2)-\frac{b}{3}\textrm{tr}(Q^3)+\frac{c}{4}(\textrm{tr}(Q^2))^2.
\end{equation*}
 \end{enumerate}
\end{defin}
\begin{remark}
\begin{enumerate}
    \item
The equation \eqref{ws:H} is derived from \eqref{newH} by testing with $\xi$. This helps us to define the appropriate penalized problem $($see \eqref{ws:Hn}$)$ and we use \eqref{ws:Hn} to derive the energy inequality \eqref{galerkin:energy}.
\item Due to the definition of uniform velocity field $u$ in \eqref{def:u}, we can also represent the energy \eqref{energy} in the following form:\begin{equation*}
E(u,Q)\stackrel{def}{=}\frac{1}{2}\int\limits_{\mcF(t)}|u|^2\,dx+\frac{1}{2} m|h'|^2+\frac{1}{2} J\omega\cdot\omega+\mathcal{E}(Q).
\end{equation*}
\end{enumerate}
\end{remark}

After introducing the definitions of the spaces and the weak solution, we are ready to announce the main result of our paper:
\begin{theorem}\label{exist:main}
Assume that $T>0$, $u_0 \in H(\Omega)$, $Q_0 \in H^1_0(\mathcal{F}_0;M)$, $\varphi_0$ is the characteristic function of $\mcS_0$ and the boundaries $\partial \Omega$, $\partial \mathcal{S}$ are of class $C^2$. Then the system \eqref{Q:fluid}--\eqref{initial} admits at least one weak solution in $[0,T]$ in the sense of Definition \ref{def:weak sol}. Moreover, we have the following energy inequality: 
\begin{equation*}
E(u,Q)(t)+\mu\int_0^t \int_{\Omega}|D u|^2\,dx\ ds 
+ \Gamma\int_0^t\int_{\mathcal{F}(t)} |H|^2\,dx\,ds
\le E(u_0,Q_0),\ \ \forall t\in[0,T],
\end{equation*}
where the total energy $E(u,Q)$ is defined in \eqref{energy}.
\end{theorem}

\section{Approximate solutions}\label{sec3}
In this section, we propose a penalized problem such that after taking appropriate limits we can recover the equations \eqref{ws:Q}--\eqref{ws:solid}. Given $u_0 \in H(\Omega)$, $Q_0 \in H^1_0(\Omega;M)$ and $\varphi_0\in L^{\infty}(\Omega)\cap \mathrm{Char}(\Omega)$, we want to find
$(u^n,Q^n,H^n,\varphi^n)$ such that
\begin{equation*}
u^n \in L^{\infty}(0,T; H(\Omega))\cap L^2(0,T;V(\Omega)\cap H^3(\Omega)),\quad Q^n\in L^{\infty}(0,T; H^1_0(\Omega;M))\cap L^2(0,T;H^2(\Omega;M)),
\end{equation*}
\begin{equation*}
H^n\in L^2(0,T;L^2(\Omega)),\quad \varphi^n \in \mbox{Char}\ (\Omega) \cap C^{0,\frac{1}{p}}(0,T;L^p(\Omega)),\ p\in[1,+\infty) 
\end{equation*}
and the following relations hold
\begin{equation}\label{ws:Qn}
\int\limits_0^T\int\limits_{\Omega} \left[- Q^n: \pa_t\psi -\Gamma H^n: \psi +(\uu\cdot \nabla)\Q: \psi -(\Sigma^n \Q-\Q\Sigma^n): \psi \right]\,dx \,dt=\int_\Om Q^n(0):\psi(0)\,dx,
\end{equation}
\begin{equation}\label{ws:Hn}
   \int_0^T\int_{\Omega} \left[ \sum_{i=1}^3 (\pa_i \Q: \pa_i\xi)+n\varphi^nQ^n\xi+ \f{\pa f_b(\Q)}{\pa Q}: \xi+H^n:\xi \right]\ dx\ dt=0
\end{equation}
\begin{multline}\label{ws:un}
\int\limits_0^T\int\limits_{\Omega} \left[-\uu\cdot \pa_t\zeta - (\uu\otimes\uu):\na\zeta+(\mu+n\varphi^n) D(\uu):D(\zeta)+\delta \nabla \Delta u^n\cdot \nabla\Delta \zeta\right]\ dx \ dt  \\
= \int_\Om u^n(0)\cdot\zeta(0)\,dx+\int\limits_0^T\int\limits_{\Omega} \left[( -(\zeta\cdot\na )Q^n:(H^n+n\varphi^nQ^n) -(\Q H^n-H^n\Q):\na \zeta \right]\ dx \ dt,
\end{multline}
\begin{equation}\label{ws:solidn}
\int\limits_0^T\int\limits_{\Omega} \varphi^n\left[\partial_t\theta+({\uu}\cdot\nabla)\theta\right]\ dx\ dt= - \int\limits_{\Omega} \varphi_0\theta(0)\ dx,
\end{equation}
for any functions $\psi\in H^1((0,T)\times\Omega;M)$,\    {\color{blue} $\xi\in L^2(0,T; H^1_0(\Om;M))$}, $\zeta\in H^1((0,T)\times \Omega)\cap L^2(0,T;V(\Omega)\cap H^3(\Omega))$ with $\zeta(T)=0$ and $\theta\in C^1([0,T]\times\Omega)$ with $\theta(T)=0$.

Let us comment on the proposed penalized problem \eqref{ws:Qn}--\eqref{ws:solidn}:
\begin{itemize}
\item We approximate the rigid bodies by very viscous fluids via the penalization term $(\mu+n\varphi^n)D(\uu):D(\zeta)$ in \eqref{ws:un}. This idea had first been introduced in the context of Incompressible Navier-Stokes-rigid body interaction problem by San Mart\'{\i}n et. al \cite{MR1870954}.
\item We regularize the fluid velocity with the $\delta \nabla \Delta u^n\cdot \nabla\Delta \zeta$ term in \eqref{ws:un} as discussed in \cite{HOST}. This regularization is useful to analyze the transport equation \eqref{ws:solidn}.
\item The approximate problem \eqref{ws:Qn}--\eqref{ws:Hn} for $Q$-tensor equation is defined on whole $\Omega$ whereas the limiting equation has to be defined only on $\mathcal{F}(t)$ (see \eqref{ws:Q}--\eqref{ws:H}). We have managed to do so by introducing the term $n\varphi^{n}Q^n$ in the equation \eqref{ws:Hn}.
\end{itemize}

\begin{remark}
The solution $(\uu,\Q,H^n,\varphi^n)$ to the approximate problem \eqref{ws:Qn}--\eqref{ws:solidn} also depends on the regularizing parameter $\delta$ (see equation\eqref{ws:un}). But, for simplicity, we write $(\uu,\Q,H^n,\varphi^n)$ instead of\\ $(u^{n,\delta},Q^{n,\delta},H^{n,\delta},\varphi^{n,\delta})$.
\end{remark}
\begin{remark}
Since $\tau= - \nabla Q\odot \nabla Q$, we observe that 
\begin{equation*}
-\mathrm{div}\tau=\mathrm{div} (\nabla Q\odot \nabla Q) = H:\nabla Q + \nabla \left(\frac{1}{2}|\nabla Q|^2+ f_b(Q)\right).
\end{equation*}
This observation helps us to introduce the term $\int\limits_{\Omega} (\zeta\cdot\na )Q^n:H^n$ in \eqref{ws:un}.
\end{remark}

 
In the following we will present an adaptation of  \cite[Theorem 1.4]{MR3274285} and \cite[Theorem 1.2]{abels2014well} regarding the  existence of weak solutions for the coupled Navier-Stokes and Q-tensor system in a bounded domain to establish the existence of solution $(\uu,\Q,H^n,\varphi^n)$ to the approximate problem \eqref{ws:Qn}--\eqref{ws:solidn}:
 \begin{theorem}\label{exist:approx}
 Assume that $u_0 \in H(\Omega)$, $Q_0 \in H^1_0(\mathcal{F}_0;M)$, $\varphi_0$ is the characteristic function of $\mcS_0$ and the boundaries $\partial \Omega$, $\partial \mathcal{S}$ are of class $C^2$. There exists a time $T_n>0$ (depending on $\operatorname{dist}(S(\varphi_0),\partial\Omega)$, $\delta$, $u_0$, $Q_0$) such that the penalized problem \eqref{ws:Qn}--\eqref{ws:solidn} has a solution $(u^{n},Q^{n},H^{n},\varphi^{n})$ which satisfies the following energy estimate for a.e. $t\in (0,T_n)$: 
\begin{multline}\label{energy:approx}
E^{app}(u^n,Q^n)(t)+\int_0^t \int_{\Omega}(\mu+n\varphi^n)|D \uu|^2\,dxds +\int_0^t\int_{\Omega}\Gamma |H^n|^2 \,dx\,ds\\ +\int_0^t\int_{\Omega}\delta |\nabla\Delta u^n|^2\,dx\,ds \le E(u_0,Q_0),
\end{multline}
where
\begin{equation*}
\begin{split}
E^{app}(u^n,Q^n)(t)&\stackrel{def}{=} \frac{1}{2}\int_{\Omega}|u^n|^2\,dx+\int_{\Omega}\left(\frac{1}{2}|\nabla Q^n|^2+ f_{b}(Q^n)\right)\,dx+n\int_{\Omega}\frac{\varphi^n}{2}|Q^n|^2(t,x)\,dx.
\end{split}
\end{equation*}
In addition, $\mcS(\varphi^n(t))$ will not touch the boundary, i.e. $\mbox{dist}(\mcS(\varphi^n(t)),\pa\Om)>0$ for all $t\in(0,T_n)$.
\end{theorem}

For convenience we drop the superscript $`n$' in the following proof of Theorem \ref{exist:approx}, i.e. we will write $(u,Q,H,\varphi)$ instead of $(u^n,Q^n,H^n,\varphi^n)$. Theorem \ref{exist:approx} is proved via a modified Galerkin method based on Dirichlet eigenfunctions of the Laplace operator for $Q$ and eigenfunctions of the Stokes operator for $u$, which was introduced in \cite{abels2014well}.  We denote the Helmholtz projection by $P:L^2(\Om,\mathbb{R}^d)\ri H(\Om)$ and the Stokes operator by $A:=-P\Delta$. We have the following well-known classical results.

\begin{lemma}\label{ortho basis: Q}
There exists an orthonormal basis $\{e_n\}_{n=1}^\infty \subset H^1_0(\Om;M)\cap H^2(\Om;M)$ of $L^2(\Om;M)$ and a non-decreasing sequence of non-negative eigenvalues $\{\lambda_n\}_{n=1}^\infty$ with $\lim\limits_{n\ri\infty} \lam_n=\infty$ such that 
\begin{equation*}
    -\Delta e_n=\lambda_n e_n\ \text{ in }\Om, \quad { e_n}=0\ \text{ on }\pa\Om.
\end{equation*}
\end{lemma}

\begin{lemma}\label{ortho basis: u}
There exists an orthonormal basis $\{v_n\}_{n=1}^\infty \subset V(\Om)\cap H^{3}(\Om,\mathbb{R}^3)\cap W^{1,\infty}(\Om,\mathbb{R}^3)$ of $H(\Om)$ and a non-decreasing sequence of positive eigenvalues $\{\omega_n\}_{n=1}^\infty$ with $\lim\limits_{n\ri\infty} \omega_n=\infty$ such that 
\begin{equation*}
Av_n=\omega_n v_n\ \text{ in }\Om.
\end{equation*}
\end{lemma}

For $k\in \mathbb{N}^+$, we define the finite-dimensional spaces
\begin{align*}
    &E_k:=\mathrm{Span}\{e_1,...,e_k\}\subset H^1_0(\Om;M)\cap H^2(\Om;M),\\
    &V_k:=\mathrm{Span}\{v_1,...,v_k\}\subset V(\Om)\cap H^{3}(\Om,\mathbb{R}^d)\cap W^{1,\infty}(\Om;\mathbb{R}^d).
\end{align*}
and the two orthogonal projections 
\begin{equation*}
    \pi_k:L^2(\Om;M)\ri E_k,\quad \mathcal{P}_k:H(\Om)\ri V_k.
\end{equation*}

We seek for the solutions of the form
\begin{equation}\label{approx-solform}
\begin{split}
    &u^{k}(x,t)=\sum\limits_{i=1}^k d_i(t)v_i(x),\quad Q^{k}(x,t)=\sum\limits_{i=1}^k q_i(t)e_i(x),
    \quad  \varphi^{k}\in \mathrm{Char}(\Om); 
\end{split}
\end{equation}
which solves the following system: for all $l=1,...,k,$

\begin{align}
\label{galerkin:Q}    & \int_{\Om} \left( \pa_t Q^{k}:e_l-\Gamma H^{k}:e_l+(u^{k}\cdot\na)Q^{k}:e_l+(Q^{k}\Sigma(u^{k})-\Sigma(u^{k})Q^{k}):e_l  \right)\,dx=0,\\
\label{galerkin:u}   & \int_\Om \bigg( \pa_t u^{k}\cdot v_l-(u^{k}\otimes u^{k}):\na v_l+(\mu+n\varphi^{k})(D(u^{k}):D(v_l))
     +\delta \nabla\Delta u^{k} \cdot\nabla\Delta v_l\\
   \nonumber  &\qquad \qquad \qquad +((H^{k}+n\varphi^{k}Q^{k}):\na Q^{k})\cdot v_l+(Q^{k}H^{k}-H^{k}Q^{k}):\na v_l\bigg)\,dx=0,\\
\label{galerkin:tau}    & \int_0^T\int_{\Omega} \varphi^{k}\left[\partial_t\theta+(u^{k}\cdot\nabla)\theta\right]\ dx\ dt= - \int_{\Omega} \varphi^{k}(x,0)\theta(x,0)\ dx, \quad \forall \ \theta\in C^1([0,T]\times \Om), \\
\label{galerkin:initial data}     & u^{k}(x,0)=\mathcal{P}_k(u_0),\quad Q^{k}(x,0)=\pi_k(Q_0),\quad \varphi^{k}(x,0)=\varphi_0,
\end{align}
where
\begin{equation}\label{def:Hk}
H^k=\pi^k\left(\Delta Q^k - \frac{\partial f_b}{\partial Q}(Q^k)-n\varphi^k Q^k\right).
\end{equation}
Observe that due to the form \eqref{def:Hk}, $H^{k}(x,t)=\sum\limits_{i=1}^k h_i(t)e_i(x)$. If we multiply equation \eqref{def:Hk} by $e_l$, integrate over $\Omega$ and do the integration by parts, we obtain:
\begin{equation}\label{galerkin:H}
\int_{\Om} \left( \na Q^{k}:\na e_l+n\varphi^{k}Q^{k}:e_l+\f{\pa f_b}{\pa Q}(Q^k):e_l+H^{k}:e_l\right)\,dx=0.
\end{equation}

\begin{proposition}\label{exist-galerkin}
Assume that $u_0 \in H(\Omega)$, $Q_0 \in H^1_0(\mathcal{F}_0;M)$, $\varphi_0$ is the characteristic function of $\mcS_0$ and the boundaries $\partial \Omega$, $\partial \mathcal{S}$ are of class $C^2$.
There exists a maximal time $T>0$ (depending on $\operatorname{dist}(S(\varphi_0),\partial\Omega)$, $\delta$, $u_0$, $Q_0$) such that the system \eqref{galerkin:Q}--\eqref{def:Hk} admits a solution $(u^{k},Q^{k},H^{k},\varphi^{k})$ which satisfies for a.e. $t\in (0,T)$, 
\begin{multline}\label{galerkin:energy}
E^{app}(u^{k},Q^{k})(t)+\int_0^t \int_{\Omega}(\mu+n\varphi^{k})|D u^{k}|^2\,dxds
+\int_0^t\int_{\Omega}\Gamma |H^{k}|^2 \,dx\,ds\\+\int_0^t\int_{\Omega} \delta |\nabla\Delta u^{k}|^2 \,dx\,ds
\le E(u_0,Q_0),
\end{multline}
where
\begin{equation*}
E^{app}(u^k,Q^k)(t)\stackrel{def}{=} \frac{1}{2}\int_{\Omega}|u^k|^2\,dx+\int_{\Omega}\left(\frac{1}{2}|\nabla Q^k|^2+ f_{b}(Q^k)\right)\,dx+n\int_{\Omega}\frac{\varphi^k}{2}|Q^k|^2(t,x)\,dx.
\end{equation*}
\end{proposition}

\begin{proof}
We consider the characteristics  $X^{k}_{s,t}:\Omega\rightarrow \Omega$ associated to the flow $u^{k}$ such that
\begin{equation}\label{char:ODE}
\frac{d}{dt}X_{s,t}^k(y)= u^{k}(t,X^k_{s,t}(y)),\quad X^k_{s,s}(y)=y \quad\forall\ (s,t,y)\in (0,T)^2\times\Omega.
\end{equation}
If $u^k\in L^2(0,T;H^3(\Omega))$, the mapping $X_{s,t}^k$ is well-defined, one-to-one and invertible. We can represent the solution to equation \eqref{galerkin:tau} in the following way:
\begin{equation}\label{rf:phi}
\varphi^k(t,y)=\varphi_0(X^k_{0,t}(y)).
\end{equation}
Moreover, we have 
\begin{equation}\label{rf:Sphi}
S(\varphi^k(t))=X^k_{0,t}(S(\varphi_0)).
\end{equation}

\textbf{\underline{Step 1: Existence of solution.}}

Let us define the following identification operators:
\begin{equation*}
I_Q: Q^k(t)\in E_k\rightarrow \mathcal{Q}(t)=(q_1(t),\cdots,q_k(t))\in \mathbb{R}^k.
\end{equation*}
and
\begin{equation*}
I_u: u^k(t)\in V_k\rightarrow \mathcal{U}(t)=(d_1(t),\cdots,d_k(t))\in \mathbb{R}^k
\end{equation*}
The equations \eqref{galerkin:Q}--\eqref{galerkin:u} form a nonlinear differential system for the function $\mathcal{Z}(t)= (\mathcal{Q}(t),\mathcal{U}(t))^{\top}$:
\begin{equation}\label{ODE1}
\mathcal{Z}'(t)+ \mathcal{A}(t,\mathcal{Z}(t))=0,\quad \mathcal{Z}(0)=(I_Q(\pi_k(Q_0)),\ I_u(\mathcal{P}_k(u_0))),
\end{equation}
where $\mathcal{A}:(t,\mathcal{Z}(t))\in (0,T)\times\mathbb{R}^{2k}\rightarrow \mathcal{A}(t,\mathcal{Z}(t))\in \mathbb{R}^{2k}$ is defined as:
\begin{equation*}
\mathcal{A}(t,\mathcal{Z}(t)) = \mathcal{B}(t,\mathcal{Z}(t)) + \mathcal{C}(t,\mathcal{Q}(t)).
\end{equation*}
Here $\mathcal{B}(t,\cdot)$ is a polynomial function in $\mathcal{Z}(t)$:
\begin{equation*}
\mathcal{B}(t,\mathcal{Z}(t))= (\mathcal{B}_1(t,\mathcal{Z}(t)) , \mathcal{B}_2(t,\mathcal{Z}(t)))^{\top},
\end{equation*}
where, for $l,m=1,2,\cdots k$:
\begin{equation*}
[ \mathcal{B}_1(t,\mathcal{Z}(t))]_l= \mathcal{U}_i(t)\mathcal{Q}_i(t)\int_{\Om}(v_i\cdot\na)e_i:e_l +\mathcal{U}_i(t)\mathcal{Q}_i(t) \int_{\Om}(e_i\Sigma(v_i)-\Sigma(v_i)e_i):e_l,
\end{equation*}
\begin{equation*}
[\mathcal{B}_2(t,\mathcal{Z}(t))]_m= -\mathcal{U}_i(t)\mathcal{U}_j(t)\int_{\Om}(v_{i}\otimes v_{j}):\na v_m+\mathcal{U}_i(t)\int_{\Om} (\mu+n\varphi^{k})(D(v_{i}):D(v_m))
     + \mathcal{U}_i(t)\int_{\Om}\delta\nabla\Delta v_{i}\cdot\nabla\Delta v_m.
\end{equation*}
The other part: $$\mathcal{C}(t,\mathcal{Q}(t))=(\mathcal{C}_1(t,\mathcal{Q}(t)),\ \mathcal{C}_2(t,\mathcal{Q}(t)))^{\top}$$ with
\begin{equation*}
[\mathcal{C}_1(t,\mathcal{Q}(t))]_l= \mathcal{G} (\mathcal{Q}(t))_j\int_{\Om}\Gamma e_j:e_l ,
\end{equation*}
\begin{equation*}
\begin{split}
[\mathcal{C}_2(t,\mathcal{Q}(t))]_m=& \mathcal{Q}_i(t)\mathcal{G} (\mathcal{Q}(t))_j\int_{\Om} ( e_ie_j-e_je_i)\cdot \nabla v_m + \mathcal{Q}_i(t) \mathcal{G} (\mathcal{Q}(t))_j \int_{\Om} (v_m\cdot\nabla)e_i\cdot e_j\\
&+\mathcal{Q}_i(t)\mathcal{Q}_j(t)\int_\Omega n\varphi^k (v_m\cdot \nabla)e_i\cdot e_j.
\end{split}
\end{equation*}
In the above expression, 
\begin{equation*}
 \mathcal{G} (\mathcal{Q}(t))=I_Q \circ  \pi_k\circ F\circ (I_Q)^{-1}(\mathcal{Q}(t)),
\end{equation*}
 where
 \begin{equation*}
 F:  E_k\subset H^2(\Om;M)\rightarrow L^2(\Om;M) \mbox{ 
 given by  }F(Q)=\Delta Q - \frac{\partial f_b(Q)}{\partial Q}-n\varphi Q.
 \end{equation*}
 The regularity of Laplace operator, properties of polynomial function $\frac{\partial f_b(Q)}{\partial Q}$ and projection operator $\pi_k$ imply that $\mathcal{G} (\mathcal{Q}(t))$ is continuous in $\mathcal{Q}(t)$. Since $\mathcal{B}(t,\cdot)$ is a polynomial function in $\mathcal{Z}(t)$, the expression \eqref{char:ODE} of $\varphi^k$ and the expression $\mathcal{C}(t,\cdot)$ imply that $\mathcal{A}(t,\mathcal{Z}(t))$ is continuous in $\mathcal{Z}(t)$. Thus, by Carath\'{e}odory's existence theorem we obtain an absolutely continuous solution $\mathcal{Z}(t)$ to the equation \eqref{ODE1} on $[0,T^*)$. We can use the energy estimate \eqref{energy:gal} to extend this interval of existence up to the contact time $T$ (see step 3).

\textbf{\underline{Step 2: Energy estimate.}}

The form of solution in \eqref{approx-solform} allows us to conclude $\partial_t Q^k, \Delta Q^k \in E_k$. We multiply \eqref{galerkin:u} by $d_{l}$ and \eqref{galerkin:Q} by $h_l$, sum it over $l=1,2,\cdots k$, integrate by parts to obtain:
\begin{multline}\label{ukQk}
\int_{\Om} \bigg(-\pa_t Q^{k}:H^k+\Gamma|H^{k}|^2\bigg)\ dx + \frac{1}{2}\frac{d}{dt}\int_\Om |u^{k}|^2 \ dx+ 
    \int_\Om \bigg( (\mu+n\varphi^{k})|D(u^{k})|^2
     +\delta |\nabla\Delta u^{k}|^2 \bigg)\,dx\\+ \frac{n}{2}\int\limits_{\pa S(\varphi^k)}|Q^k|^2 (u^k\cdot \nu)\ d\sigma=0.
\end{multline}
Here we have used the following identities:
\begin{equation*}
\int_\Om \varphi^{k}(Q^{k}:\na Q^{k})\cdot u^k \,dx=\f12\int\limits_{\pa S(\varphi^k)}|Q^k|^2 (u^k\cdot \nu)\ d\sigma, \quad \mbox{as}\quad \operatorname{div}u^k=0,
\end{equation*}
and 
\begin{equation*}
-(Q^k\Sigma^k-\Sigma^k Q^k)\cdot H^k+(Q^kH^k-H^kQ^k):\na u^k=0,
\end{equation*}
which follows from
\cite[Lemma 2.1]{abels2014well}.
Observe that $\pa_t Q^k=\sum_{i=1}^k q_i'(t)e_i(x) \in E_k$. If we multiply \eqref{galerkin:H} by $q_l'$ and sum it over $l=1,2\cdots k$, we have
\begin{equation*}
\frac{1}{2}\frac{d}{dt}\int_{\Om} |\na Q^{k}|^2\ dx +\int_{\Om} \left(\frac{n}{2}\varphi^{k}\pa_t|Q^{k}|^2+\f{\pa f_b(Q^{k})}{\pa Q^k}:\pa_t Q^{k}+H^{k}:\pa_t Q^k\right)\,dx=0.
\end{equation*}
Moreover, we can write it in the following way:
\begin{equation}\label{subs in Hk}
    \begin{split}
\int\limits_{\Omega} -\pa_tQ^k : H^k \,dx &= \frac{d}{dt}\int\limits_{\Omega} \left(\f12|\na Q^k|^2+f_b(Q^k)\right)\ dx +\int_{\Om} \frac{n}{2}\varphi^{k}\pa_t|Q^{k}|^2\ dx.
    \end{split}
\end{equation}
If we substitute \eqref{subs in Hk} in the equation \eqref{ukQk}, we obtain
\begin{multline*}
\frac{d}{dt}\int\limits_{\Omega} \left(\f12|u^k|^2+\f12|\na Q^k|^2+f_b(Q^k)\right)\ dx +\int_{\Om} \Gamma|H^{k}|^2\ dx + 
    \int_\Om \bigg( (\mu+n\varphi^{k})|D(u^{k})|^2
     +\delta |\nabla\Delta u^{k}|^2 \bigg)\,dx\\+ \int_{S(\varphi^k)} \frac{n}{2}\pa_t|Q^{k}|^2\ dx+\f{n}2\int\limits_{\pa S(\varphi^k)}|Q^k|^2 (u^k\cdot \nu)\ d\sigma=0.
\end{multline*}
We use Reynolds transport theorem to rewrite the above relation as:
\begin{multline}\label{energy:gal}
\frac{d}{dt}\int\limits_{\Omega} \left(\f12|u^k|^2+\f12|\na Q^k|^2+f_b(Q^k)\right)\ dx+\frac{d}{dt}\int\limits_{S(\varphi^k)}\frac{n}{2}|Q^{k}|^2\ dx \\ +\int_\Om \bigg( \Gamma|H^{k}|^2 + 
(\mu+n\varphi^{k})|D(u^{k})|^2
     +\delta |\nabla\Delta u^{k}|^2 \bigg)\,dx=0.
\end{multline}
Thus, we have established the energy estimate \eqref{galerkin:energy} in the Galerkin level.

\textbf{ \underline{Step 3: Body away from boundary.}} We want to show that there exists $T$  such that
\begin{equation*}
\operatorname{dist}(S(\varphi^k(t)),\partial\Omega)>0.
\end{equation*}
Let us assume that initially the body is away from the boundary, i.e,
\begin{equation*}
\operatorname{dist}(S(\varphi_0),\partial\Omega)= 2\alpha>0.
\end{equation*}
Observe that 
\begin{equation*}
X^k_{0,t}(y)=X^k_{0,0}(y)+\int\limits_0^t \frac{\partial}{\partial s}X^k_{0,s}(y)\ ds=y+\int\limits_0^t \frac{\partial}{\partial s}X^k_{0,s}(y)\ ds.
\end{equation*}
We know from \eqref{rf:Sphi} that for any $y\in S(\varphi_0)$, $X^k_{0,t}(y)\in S(\varphi^k(t))$. Now for any $y\in S(\varphi_0)$ and $x\in\partial\Omega$, we can write:
\begin{equation*}
|x-y|\leq |x-X^k_{0,t}(y)|+|X^k_{0,t}(y)-y|.
\end{equation*}
We can use the triangle inequality, the definition \eqref{char:ODE} to write
\begin{equation*}\label{X:dist}
|X^k_{0,t}(y)-x|\geq |x-y| - |X^k_{0,t}(y)-y|\geq 2\alpha - |\int\limits_0^t \frac{\partial}{\partial s}X^k_{0,s}(y)\ ds|\\ = 2\alpha-|\int\limits_0^t u^k(s,X^k_{0,s}(y))\ ds| .
\end{equation*}
Now using the embedding $H^3(\Omega)\hookrightarrow L^{\infty}(\Omega)$ and the estimate \eqref{energy:gal}, we deduce
\begin{equation*}
|\int\limits_0^t u^k(s,X^k_{0,s}(y))\ ds|\leq \int\limits_0^t \|u^k(s)\|_{L^{\infty}(\Omega)}\ ds\leq  \|\nabla\Delta u^k\|_{L^2(0,T;L^{2}(\Omega))}\sqrt{T} \leq \frac{\sqrt{T}}{\delta}E(u_0,Q_0).
\end{equation*}
In particular, by using above two estimates, we can choose $T$ (depending on $\operatorname{dist}(S(\varphi_0),\partial\Omega)$, $\delta$, $u_0$, $Q_0$) such that
\begin{equation*}
\operatorname{dist}(S(\varphi^k(t)),\partial\Omega)\geq\alpha>0.
\end{equation*}
\end{proof}
\subsection{Convergence of the Galerkin Scheme}
In Proposition \ref{exist-galerkin}, we have already constructed a solution $(u^{k},Q^{k},H^{k},\varphi^{k})$ to the problem \eqref{galerkin:Q}--\eqref{galerkin:H}. In this section, we plan to establish Theorem \ref{exist:approx} by passing to the limit in \eqref{galerkin:Q}--\eqref{galerkin:H} as $k\rightarrow\infty$ to recover the solution of \eqref{ws:Qn}--\eqref{ws:solidn}, i.e. of the $(n,\delta)$ approximation system.
\begin{proof} [Proof of \cref{exist:approx}]
 We have the following bounds regarding the space regularity from the energy estimate \eqref{galerkin:energy}:
 \begin{equation}\label{recall:energy}
 \|u^k\|_{L^2(0,T; V(\Omega)\cap H^3(\Omega))} + \|u^k\|_{L^{\infty}(0,T; H(\Omega))} + \|Q^k\|_{L^2(0,T; H^2(\Omega))} + \|Q^k\|_{L^{\infty}(0,T; H^1(\Omega))} \leq E(u_0,Q_0).
 \end{equation}
 We want to divide our proof in several steps.

 \underline{\textbf{Step 1: Time regularity.}}
 Let us start with the derivation of the regularity in time. Let us start with the estimate for the quantity $\partial_t u^k$. Observe that we have from the equation \eqref{galerkin:u} that for all $v_l \in V_k$:
 \begin{multline}\label{recalleq:uk}
\left| \int_\Om  \pa_t u^{k}\cdot v_l \ dx\right| \leq \int_\Om \bigg(|(u^{k}\otimes u^{k}):\na v_l|+|(\mu+n\varphi^{k})(D(u^{k}):D(v_l))|
     +\delta |\nabla\Delta u^{k} \cdot\nabla\Delta v_l|\\
    +|((H^{k}+n\varphi^{k}Q^{k}):\na Q^{k})\cdot v_l|+|(Q^{k}H^{k}-H^{k}Q^{k}):\na v_l|\bigg)\,dx.
   \end{multline}
   Using the H\"{o}lder's inequality and the embedding $H^2(\Omega)\hookrightarrow L^{\infty}(\Omega)$, we can have the following estimates 
   \begin{equation}\label{ut1}
  \int_\Om |(u^{k}\otimes u^{k})(t):\na v_l|\ dx \leq \|u^{k}(t)\|_{L^2(\Omega)}\|u^{k}(t)\|_{L^2(\Omega)}\|\nabla v_{l}\|_{L^{\infty}(\Omega)} \leq \|u^{k}(t)\|_{L^2(\Omega)}\|u^{k}(t)\|_{L^2(\Omega)}\|v_{l}\|_{H^{3}(\Omega)},
   \end{equation} 
   \begin{equation}\label{ut2}
   \int_\Om |(\mu+n\varphi^{k})(D(u^{k}(t)):D(v_l))| \ dx \leq C\|\nabla u^{k}(t)\|_{L^2(\Omega)}\|v_{l}\|_{H^{1}(\Omega)},
   \end{equation}
\begin{equation}\label{ut2a}
   \int_\Om \delta |\nabla\Delta u^{k} \cdot\nabla\Delta v_l|\ dx \leq \delta\|\nabla\Delta u^{k}\|_{L^2(\Omega)} \|\nabla\Delta v_l\|_{L^2(\Omega)} \leq C\|\nabla\Delta u^{k}\|_{L^2(\Omega)}\| v_l\|_{H^3(\Omega)},
\end{equation}
   \begin{gather}\label{ut3}
  \int_\Om |((H^{k}+n\varphi^{k}Q^{k})(t):\na Q^{k}(t))\cdot v_l|\ dx  \leq \int_\Om |(H^{k}(t):\na Q^{k}(t))\cdot v_l|\ dx + n\int_{\Om} |(Q^{k}(t):\na Q^{k}(t))\cdot v_l|\ dx\notag \\ \leq \|H^{k}(t)\|_{L^2(\Omega)}\|\na Q^{k}(t)\|_{L^2(\Omega)}\|v_{l}\|_{L^{\infty}(\Omega)} + n\| Q^{k}(t)\|_{L^2(\Omega)}\|\na Q^{k}(t)\|_{L^2(\Omega)}\|v_{l}\|_{L^{\infty}(\Omega)}\notag\\ 
  \leq  C\|H^{k}(t)\|_{L^2(\Omega)}\|\na Q^{k}(t)\|_{L^2(\Omega)}\|v_{l}\|_{H^{2}(\Omega)} + C\| Q^{k}(t)\|_{L^2(\Omega)}\|\na Q^{k}(t)\|_{L^2(\Omega)}\|v_{l}\|_{H^{2}(\Omega)}
  \end{gather}
  and 
  \begin{equation}\label{ut4}
   \int_\Om |Q^k H^k:\nabla v_l| \ dx  \leq \|Q^{k}(t)\|_{L^2(\Omega)}\|H^{k}(t)\|_{L^2(\Omega)}\|\nabla v_{l}\|_{L^{\infty}(\Omega)} \leq \|Q^{k}(t)\|_{L^2(\Omega)}\|H^{k}(t)\|_{L^2(\Omega)}\|v_{l}\|_{H^{3}(\Omega)}.
   \end{equation}
   Thus combination of the equation \eqref{recalleq:uk} and the estimates \eqref{ut1}--\eqref{ut4} yields: 
   \begin{equation*}
   \left| \int_\Om  \pa_t u^{k}\cdot v_l \ dx\right| \leq Ca_n(t)\|v_l\|_{H^3(\Omega)}\mbox{ for all }v_l\in V_k,
   \end{equation*}
   where $a_n(t)$ is defined by 
   \begin{multline}\label{eq:an}
   a_n(t):= \|u^{k}(t)\|_{L^2(\Omega)}\|u^{k}(t)\|_{L^2(\Omega)} + \|\nabla u^{k}(t)\|_{L^2(\Omega)} + \|\nabla\Delta u^{k}\|_{L^2(\Omega)}+ \|H^{k}(t)\|_{L^2(\Omega)}\|\na Q^{k}(t)\|_{L^2(\Omega)}\\+ \|Q^{k}(t)\|_{L^2(\Omega)}\|\na Q^{k}(t)\|_{L^2(\Omega)}+ \|Q^{k}(t)\|_{L^2(\Omega)}\| H^{k}(t)\|_{L^2(\Omega)}.
   \end{multline}
   Due to the fact that $\int_\Om  \pa_t u^{k}\cdot v_l=0$, for all $v_l \perp V_k$, we obtain
   \begin{equation*}
   \left| \int_\Om  \pa_t u^{k}\cdot v_l \ dx\right| \leq C a_n(t)\|v_l\|_{H^3(\Omega)}, \quad \forall \ v_l \in H^3(\Omega),
   \end{equation*}
   which implies 
     \begin{equation*}
     \|\pa_t u^{k}(t)\|_{(H^3(\Omega))'} \leq Ca_n(t), \mbox{ for almost all }t\in (0,T).
   \end{equation*}
   It is clear from the estimate \eqref{recall:energy} that $a_n(t)\in L^2(0,T)$ and we have the following time regularity:
   \begin{equation}\label{treg:u}
   \pa_t u^{k} \in L^2(0,T;(H^3(\Omega))').
   \end{equation}
   Next we want to estimate the time regularity of the term $\pa_t Q^k$. Note that we have from the equation \eqref{galerkin:Q} that for all $e_l \in V_k$:
 \begin{equation}\label{recalleq:Qk}
 \int_{\Om}  \pa_t Q^{k}:e_l \ dx = \int_{\Om} \left(\Gamma H^{k}:e_l-(u^{k}\cdot\na)Q^{k}:e_l-(Q^{k}\Sigma(u^{k})-\Sigma(u^{k})Q^{k}):e_l  \right)\,dx
 \end{equation}
 By H\"{o}lder inequality,
\begin{equation}\label{Qt1}
    \begin{split}
       \left| \int_{\Omega} \Gamma H^{k}:e_l\,dx\right|&\leq C\|H^k\|_{L^2(\Omega)}\|e_l\|_{L^2(\Omega)},\\
       \left| \int_{\Omega} (u^{k}\cdot\na)Q^{k}:e_l\,dx\right|&\leq \|u^k\|_{L^6(\Omega)}\|\na Q^k\|_{L^2(\Omega)}\|e_l\|_{L^3(\Omega)}\\
        &\leq C\|u^k\|_{H^1(\Om)}\|Q^k\|_{H^1(\Om)}\|e_l\|_{H^1(\Om)},\\
        \left| \int_{\Omega} (\Sigma^k Q^k-Q^k\Sigma^k):e_l \,dx\right|&\leq \|\na u^k\|_{L^2(\Om)}\|Q^k\|_{L^6(\Om)}\|e_l\|_{L^3(\Om)}\\
        &\leq C\|u^k\|_{H^1(\Om)}\|Q^k\|_{H^1(\Om)}\|e_l\|_{H^1(\Om)}.
    \end{split}
\end{equation}
Thus combination of the equation \eqref{recalleq:Qk} and the estimate \eqref{Qt1} yields: 
   \begin{equation*}
   \left|\int_\Om  \pa_t Q^{k}\cdot e_l\ dx\right| \leq Cb_n(t)\|e_l\|_{H^1_0(\Omega)}\mbox{ for all }e_l\in E_k,
   \end{equation*}
   where $b_n(t)$ is defined by 
   \begin{equation}\label{eq:bn}
   b_n(t):= \|H^{k}(t)\|_{L^2(\Omega)} + \| u^{k}(t)\|_{H^1(\Omega)}  \|Q^{k}(t)\|_{H^1(\Omega)}+ \|u^{k}(t)\|_{H^1(\Omega)}\| Q^{k}(t)\|_{H^1(\Omega)}.
   \end{equation}
   
   Due to the fact that $\int_\Om  \pa_t Q^{k}\cdot e_l=0$, for all $e_l \perp E_k$, we obtain
   \begin{equation*}
   \left|\int_\Om  \pa_t Q^{k}\cdot e_l\ dx\right| \leq C b_n(t)\|e_l\|_{H^1_0(\Omega)}, \quad \forall \ e_l \in H^1_0(\Omega),
   \end{equation*}
   which implies 
     \begin{equation*}
     \|\pa_t Q^{k}(t)\|_{H^{-1}(\Omega)} \leq Cb_n(t), \mbox{ for almost all }t\in (0,T).
   \end{equation*}
   It is clear from the estimate \eqref{recall:energy} that $b_n(t)\in L^2(0,T)$ and we have the following time regularity:
   \begin{equation}\label{treg:Q}
   \pa_t Q^{k} \in L^2(0,T;H^{-1}(\Omega)).
   \end{equation}
 \underline{\textbf{Step 2: Compactness and convergence.}} Now we fix $n>0$, $\delta>0$ and let $k\rightarrow\infty$. We know from the boundedness of $(u^k,Q^k,H^k)$ obtained in \eqref{recall:energy} and from the weak compactness result:
 \begin{align}
 u^k\rightarrow u \mbox{ weakly in }L^2(0,T; H^3(\Omega)) \mbox{ and weakly}^{*} \mbox{ in }L^{\infty}(0,T; H(\Omega)),\label{con:uk}\\
 Q^k \rightarrow Q \mbox{ weakly in }L^2(0,T; H^2(\Omega))  \mbox{ and weakly}^{*} \mbox{ in }L^{\infty}(0,T; H^1(\Omega)),\label{con:Qk}\\
 H^k\rightarrow H \mbox{ weakly in }L^2(0,T; L^2(\Omega)), \label{con:Hk}\\
 \varphi^k\rightarrow \varphi \mbox{ weakly}^{*} \mbox{ in }L^{\infty}(0,T; L^{\infty}(\Omega)),\label{con:phik}
 \end{align}
for some limit function $u\in L^2(0,T;H^3(\Omega))\cap L^\infty(0,T;H(\Omega))$, $Q\in L^2(0,T; H^2(\Omega)) \cap L^\infty(0,T;H^1(\Omega))$, $H\in L^2([0,T]\times \Om)$, $\varphi\in \mbox{Char}(Q)$.
Moreover, we can use Lions-Aubin-Simon Lemma \cite{simon1986compact} with the energy estimate \eqref{recall:energy} and boundedness of $\pa_t u^k$, $\pa_t Q^k$  in \eqref{treg:u}, \eqref{treg:Q} respectively to conclude the strong convergences:
\begin{align}
 u^k\rightarrow u \mbox{ strongly in }L^2(0,T; L^2(\Omega)) \mbox{ and in }C(0,T; H^{-1}(\Omega)),\label{scon:uk}\\
 Q^k \rightarrow Q \mbox{ strongly in }L^2(0,T; H^{2-\varepsilon}(\Omega))\cap L^p((0,T)\times\Omega)   \mbox{ and in }C(0,T; L^2(\Omega)),\label{scon:Qk}
 \end{align}
 for a fixed $\varepsilon > 0$ and for all $p\in (1,6)$.

  \underline{\textbf{Step 3: Transport equation.}}  The weak convergence of $(u^k, \varphi^k)$ to $(u,\varphi)$ in the sense of \eqref{con:uk}, \eqref{con:phik} and the compactness result of transport equation \cite[Theorem II.4]{DiPerna1989} (see also \cite[Lemma 5.1]{MR1870954}), 
provide
\begin{equation}\label{conv:phi}
    \varphi^k \rightarrow \varphi \mbox{ strongly in }C([0,T];L^p(\Omega)),\quad 1\leq p<\infty,
   \end{equation}
   We consider the characteristics  $X^k_{s,t}:\Omega\rightarrow \Omega$ defined as in \eqref{char:ODE} associated to the flow $u^{k}$. We can represent $\varphi^k$ as in \eqref{rf:phi}. We know from the properties of the family of the characteristics $X^k_{s,t}$ \cite[Corollary 5.2]{MR1870954} that as $k\rightarrow\infty$:
   \begin{equation}
 X^k_{s,t}(y)\rightarrow X_{s,t}(y) \mbox{ strongly in }C^{0,\alpha}([0,T]^2;C^1(\Omega)),\ \alpha< 1,
\end{equation}
where $X_{s,t}(y)$ is the unique solution of the problem
 \begin{equation*}
\frac{d}{dt}X_{s,t}(y)=u(t,X_{s,t}(y)),\ \ X_{s,s}(y)=y,\ \ \forall\ (s,t,y)\in(0,T)^2\times \Omega.
\end{equation*}
Moreover, we know from \cite[Theorem 2.4, Remark 2.4, Pages 41-42]{Lions1996} (see also \cite[Theorem 3.1, Corollary 5.3]{MR1870954}) that $\varphi\in \mbox{Char}(Q)\cap C(0,T;L^1(\Omega))$  satisfies the following
\begin{equation}\label{ode:phi}
 \frac{\partial \varphi}{\partial t}+ \operatorname{div}(\varphi u)=0 \quad \mbox{in}\quad \mathcal{D}'((0,T)\times\Omega),\quad
 \varphi(0)=\varphi_0 \quad \text{a.e. in}\quad \Omega.
 \end{equation}

In other words, we have established the weak formulation of transport equation  \eqref{ws:solidn} as a limit equation of \eqref{galerkin:tau} when $k\rightarrow\infty$. 
\begin{remark}
We can represent the solution to equation \eqref{ws:solidn} in the following way:
\begin{equation*}
\varphi(t,y)=\varphi_0(X_{0,t}(y)).
\end{equation*}
\end{remark}
  
  \underline{\textbf{Step 4: Equation for u.}} Let us replace $v_l$ in equation \eqref{galerkin:u} by any $\zeta\in C^1 ([0,T]; V(\Omega)\cap H^3(\Omega))$ of the form $$\zeta(x,t)=\sum\limits_{i=1}^k d^{i}(t)v_i(x)$$ and obtain the following equation:
\begin{multline}\label{uk1}
\int_\Om \bigg( \pa_t u^{k}\cdot \zeta-(u^{k}\otimes u^{k}):\na \zeta+(\mu+n\varphi^{k})(D(u^{k}):D(\zeta))
     +\delta \nabla\Delta u^{k} \cdot\nabla\Delta \zeta\\ +((H^{k}+n\varphi^{k}Q^{k}):\na Q^{k})\cdot \zeta+(Q^{k}H^{k}-H^{k}Q^{k}):\na \zeta\bigg)\,dx=0
\end{multline}
We choose $d^{i}(T)=0$, integrate the equation \eqref{uk1} with respect to time and do integration by parts in the first term to obtain:
\begin{multline}\label{uk2}
\int_0^T \int_\Om \bigg( - u^{k}\cdot \pa_t\zeta-(u^{k}\otimes u^{k}):\na \zeta+(\mu+n\varphi^{k})(D(u^{k}):D(\zeta))
     +\delta \nabla\Delta u^{k} \cdot\nabla\Delta \zeta\\ +((H^{k}+n\varphi^{k}Q^{k}):\na Q^{k})\cdot \zeta+(Q^{k}H^{k}-H^{k}Q^{k}):\na \zeta\bigg)\,dx=\int_{\Om} u^k(0)\cdot \zeta (0)\ dx.
\end{multline}
We use the weak and strong convergences of $u^k$ (see \eqref{con:uk} and \eqref{scon:uk} respectively) to pass the limit $k\rightarrow\infty$ in the first four terms in \eqref{uk2}. Moreover, we use the weak and strong convergences of $Q^k$ (see \eqref{con:Qk} and \eqref{scon:Qk} respectively), weak convergence of $H^k$ (see \eqref{con:Hk}) to pass the limit $k\rightarrow\infty$ in the fifth and sixth terms in \eqref{uk2}. Thus we can conclude 
\begin{multline*}
\int_0^T \int_\Om \bigg( - u\cdot \pa_t\zeta-(u\otimes u):\na \zeta+(\mu+n\varphi)(D(u):D(\zeta))
     +\delta \nabla\Delta u \cdot\nabla\Delta \zeta\\ +((H+n\varphi^{k}Q):\na Q)\cdot \zeta+(QH-HQ):\na \zeta\bigg)\,dx=\int_{\Om} u(0)\cdot \zeta (0)\ dx.
\end{multline*}
We can use a density argument to conclude that the above equation also holds for all $\zeta\in H^1([0,T]\times\Omega)\cap L^2(0,T;V(\Omega)\cap H^3(\Omega))$ such that $\zeta(T)=0$ and we have established the weak formulation \eqref{ws:un} as a limit equation of \eqref{galerkin:u}.
  
   \underline{\textbf{Step 5: Equation for H and Q.}} Let us replace $e_l$ in equation \eqref{galerkin:Q} by any $\psi\in C^1 ([0,T]; H^1(\Omega;M))$ of the form $$\psi(x,t)=\sum\limits_{i=1}^k q^{i}(t)e_i(x),\quad q^i(T)=0,$$ and obtain the following equation:
   \begin{equation}\label{Qk1}
   \int_{\Om} \left( \pa_t Q^{k}:\psi-\Gamma H^{k}:\psi+(u^{k}\cdot\na)Q^{k}:\psi+(Q^{k}\Sigma(u^{k})-\Sigma(u^{k})Q^{k}):\psi  \right)\,dx=0.
   \end{equation}
   Integrate the equation \eqref{Qk1} with respect to time and do integration by parts in the first term to obtain:
    \begin{multline}\label{Qk2}
   \int_0^T\int_{\Om} \left( - Q^{k}:\pa_t\psi-\Gamma H^{k}:\psi+(u^{k}\cdot\na)Q^{k}:\psi+(Q^{k}\Sigma(u^{k})-\Sigma(u^{k})Q^{k}):\psi  \right)\,dx \\=\int_{\Om} Q^k(0)\cdot \psi (0)\ dx.
   \end{multline}
   We can use the weak convergences of $u^k$, $Q^k$, $H^k$, $\varphi^k$ (see \eqref{con:uk}--\eqref{con:phik}) and strong convergence of $Q^k$ (see \eqref{scon:Qk}) to pass the limit $k\rightarrow\infty$ in \eqref{Qk2} and to conclude 
     \begin{equation*}
   \int_0^T\int_{\Om} \left( - Q:\pa_t\psi-\Gamma H:\psi+(u\cdot\na)Q:\psi+(Q\Sigma(u)-\Sigma(u)Q):\psi  \right)\,dx =\int_{\Om} Q(0)\cdot \psi (0)\ dx.
   \end{equation*}
   Utilizing a density argument we conclude that the above equation also holds for all $\psi\in H^1([0,T]\times\Omega;M)$ with $\psi(T)=0$ and we have established the weak formulation \eqref{ws:Qn} as a limit equation of \eqref{galerkin:Q}. Similarly the weak formulation \eqref{ws:Hn} as a limit equation of \eqref{galerkin:H} follows by the weak convergences (see \eqref{con:uk}--\eqref{con:phik}), strong convergence of $Q^k$ (see \eqref{scon:Qk}) and the expression of $H^k$ (see \eqref{def:Hk}).

   Finally, after passing to the limit $k\ri\infty$ in \eqref{galerkin:energy}, we obtain the energy inequality \eqref{energy:approx}. Moreover, we can follow the same deduction as in the Step 3 of the proof of Proposition \ref{exist-galerkin} to conclude that we can choose $T=T_{n}$ (depending on $\operatorname{dist}(S(\varphi_0),\partial\Omega)$, $\delta$, $u_0$, $Q_0$) such that
\begin{equation*}
\operatorname{dist}(S(\varphi(t)),\partial\Omega)>0 \text{ for }t\in[0,T).
\end{equation*}
Thus, the process of passing $k\ri\infty$ is complete.
\end{proof}
 \section{Proof of Theorem \ref{exist:main}}\label{sec4}
In this section we prove the global existence of a weak solution to the system \eqref{Q:fluid}--\eqref{initial}. We will first establish the local existence of a weak solution before the colloid touches the boundary and then extend the solution after the time of collision. For the local existence we pass to the limit in the approximated system \eqref{ws:Qn}--\eqref{ws:solidn} firstly as  $n\rightarrow \infty$ and then as $\delta\rightarrow 0$. 

\subsection{Pass to the limit as $n\rightarrow \infty$}\label{subsec: n limit}

Now we fix $\delta>0$ and let $n\rightarrow\infty$. Let $T=T(\delta,u_0,Q_0,\varphi_0)$ be a fixed constant such that the system \eqref{ws:Qn}--\eqref{ws:solidn} admits a solution on the time interval $[0,T)$ such that $\mbox{dist}(\mcS(\varphi^n(t)),\partial\Omega)>0$ for $t\in[0,T)$. Note that from the proof of Theorem \ref{exist:approx}, $T$ is independent of $n$. Let $(u^n,Q^n,H^n,\varphi^n)$ denote the solution of the system.  We have the boundedness of $u^n$, $Q^n$ and $H^n$ in appropriate norms from the inequality \eqref{energy:approx} along with the existence of solution $(\uu,\Q,H^n,\varphi^n)$ from Theorem \ref{exist:approx}. Using Banach-Alaoglu Theorem, we have up to a subsequence
 \begin{align}
 u^n\rightarrow u \mbox{ weakly in }L^2(0,T; H^3(\Omega)) \mbox{ and weakly}^{*} \mbox{ in }L^{\infty}(0,T; H(\Omega)),\label{con:un}\\
 Q^n\rightarrow Q  \mbox{ weakly}^{*} \mbox{ in }L^{\infty}(0,T; H^1(\Omega)),\label{con:Qn}\\
 H^n\rightarrow H \mbox{ weakly in }L^2(0,T; L^2(\Omega)), \label{con:Hn}\\
 \varphi^n\rightarrow \varphi \mbox{ weakly}^{*} \mbox{ in }L^{\infty}(0,T; L^{\infty}(\Omega)),\label{con:phin}
 \end{align}
for some limit functions $u\in L^2(0,T;H^3(\Omega))\cap L^\infty(0,T;H(\Omega))$, $Q\in L^\infty(0,T;H^1(\Omega))$, $H\in L^2((0,T)\times \Om)$, $\varphi\in \mbox{Char}(Q)$.

\underline{\textbf{Step 1: Passing to the limit in the transport equation.}} We want to pass the limit $n\rightarrow\infty$ in the equation \eqref{ws:solidn} and to do so the first aim is to analyze the limit of $\varphi^n$. We consider again the mapping $X_{s,t}^n:\Omega\rightarrow\Omega$ that satisfies 
\begin{equation*}
    \frac{d}{dt}X_{s,t}^n(y)=u^n(t,X^n_{s,t}(y)),\ \ X^n_{s,s}(y)=y,\ \ \forall (s,t,y)\in(0,T)^2\times \Omega.
\end{equation*}
Since $u^n\in L^2(0,T;H^3(\Omega))\subset L^2(0,T;C^{1,\frac12}(\Omega))$, the mapping $X^n_{s,t}$ is well-defined, one-to-one and invertible. Then by \eqref{ws:solidn} we have 
\begin{equation*}
 \varphi^n(t,y)=\varphi_0(X^n_{t,0}(y)),\quad \mathcal{S}(\varphi^n(t))=X_{0,t}^n(\mathcal{S}(\varphi_0)).    
\end{equation*}


We can follow the same argument as in Step 3 of the proof of \cref{exist:approx} to conclude that we can pass the limit $n\rightarrow\infty$ in the equation \eqref{ws:solidn} and obtain \eqref{ws:solid}.

 Moreover, we can utilize the second term in the left-hand side of energy inequality \eqref{energy:approx} and the convergence results \eqref{con:un}, \eqref{conv:phi} corresponding to $\uu$, $\varphi^n$ respectively to conclude that
\begin{equation*}
\varphi Du =0 \quad \text{a.e. in }\ Q.
\end{equation*}
Thus the velocity field $u\in L^2(0,T;K(\varphi))$. 
 
The following result is proved separately in \cite[Lemma 3.3]{HOST} and \cite[Lemma 6.1]{MR1870954}, where the arguments can adapt to our problem easily.  

\begin{lemma}\label{lemma 6.1 2d paper}
(\cite[Lemma 6.1]{MR1870954}) For any $\sigma>0$ there exists $n_0$ (depending on $\sigma,\,\delta$) such that 
\begin{equation*}
    \mathcal{S}(\varphi^n(t))\subset \mcS_{\sigma}(\varphi(t)),\quad\mathcal{S}(\varphi(t))\subset \mcS_{\sigma}(\varphi^n(t)) 
\end{equation*}
for all $n\geq n_0$, for all $t\in [0,T)$.
\end{lemma}
 
 \underline{\textbf{Step 2: Strong convergences.}} We need to establish some strong convergence results regarding the Landau-de Gennes tensor $Q$ and the velocity $u$ to proceed further in the passing to the limit for the equations \eqref{ws:Qn}--\eqref{ws:un}. First of all we point out that utilizing the energy estimate \eqref{energy:approx}, weak convergence result \eqref{con:Qn}, and Lemma \ref{lemma 6.1 2d paper} one can immediately verify that 
 \begin{equation}\label{H=Q=0 in solid}
 Q=0\quad \text{ a.e. in }\mcS(\varphi(t)),    
 \end{equation}
 which implies 
 \begin{equation*}
     Q(t,\cdot)\in H_0^1(\Omega\setminus \mcS(\varphi(t))) \text{ for a.e. }t\in[0,T).
 \end{equation*}
 
 We begin with the proof of strong convergence of $\{Q^n\}$ in $L^2(0,T; H^1(\mathcal{F}(t)))$.
 \begin{lemma}\label{strong con lemma:Q}
 For any $\sigma>0$, the sequence $\{Q^n\}$ in \eqref{con:Qn} converges strongly to $Q$ in $L^2(0,T; H^1(\widetilde{\mathcal{F}}_\sigma(t)))$, where $\widetilde{\mathcal{F}}_\sigma(t)=\Omega \setminus \mathcal{S}_\sigma(\varphi(t))$.
\end{lemma}
\begin{proof}
 By Lemma \ref{lemma 6.1 2d paper}, there exists $n_0>0$ such that 
\begin{equation}\label{phi^n belong to nbh of phi}
    \mcS(\varphi^n(t))\subset \mcS_{\frac{\sigma}{2}}(\varphi(t)),\;\;\forall t\in[0,T],\,\forall\ n > n_0.
\end{equation}
Moreover, since the map $X_{0,t}(x)$ is H\"{o}lder in $t$, there is a positive integer $N_\sigma$ such that we can divide the interval $[0,T]$ into $N_\sigma$ sub-intervals and find $N_\sigma$ corresponding subsets of $\Omega$, denoted by $\mcS_1,\,\mcS_2,\,...,\,\mcS_{N_\sigma}$ (here $\mcS_j$ is independent of $t$), that satisfy
\begin{align}
    \nonumber&\qquad [0,T]=\bigcup\limits_{j=1}^{N_\sigma} I_j,\quad I_j:=[(j-1)\tau,j\tau],\quad \tau=\frac{T}{N_\sigma},\\
    & \label{nbh of phi belong to sj}  \mcS_{\frac{\sigma}{2}}(\varphi(t))\subset \mcS_j\subset\mcS_{\sigma}(\varphi(t)), \quad \text{for any }j=1,...,N_\sigma\text{ and }t\in I_j. 
\end{align}
Due to the regularity assumption on $\partial\mcS(\varphi(t))$, we have 
\begin{equation}
    \mathcal{L}(\mcS_j \setminus \mcS(\varphi(t)))\leq\mathcal{L}(\mcS_\sigma(\varphi(t)) \setminus \mcS(\varphi(t))) \leq C\sigma, \quad \forall j=1,...,N_\sigma,\;t\in I_j,
\end{equation}
where $C$ is a constant independent of $t$ and $\sigma$, but only depends on $|\partial \mcS(\varphi(0))|$. 

For each $j$, we examine the solution $(u^n,Q^n,H^n,\varphi^n)$ of \eqref{ws:Qn}--\eqref{ws:solidn} on the cylinder $I_j\times (\Omega\setminus\mathcal{S}_j)$. We take the test function $\xi$ in \eqref{ws:Hn} such that $\mbox{supp}(\xi)\subseteq I_j\times (\Omega\setminus\mathcal{S}_j)$. When $n>n_0$, since $\mathcal{S}(\varphi^n(t))\subset S_j$ for $t\in I_j$, $\xi\varphi^n=0$ in $I_j\times (\Omega\setminus\mathcal{S}_j)$. Further we have
\begin{equation} \label{Delta Qn represented by Hn in Ij}
    \Delta Q^n=\frac{\partial f_b}{\pa Q}(Q^n)+H^n \quad \text{ in }\mathcal{D}'(I_j\times (\Omega\setminus\mathcal{S}_j)).
\end{equation}
Utilizing the energy estimate \eqref{energy:approx} we have 
\begin{equation}\label{uniform bound for Qn and Hn in Ij}
    \|Q^n\|_{L^\infty(0,T;H^1(\Omega))}+\|H^n\|_{L^2([0,T]\times \Omega)}\leq C.
\end{equation}
Since $|\frac{\partial f_b}{\pa Q}(Q^n)|\leq C(|Q^n|^3+1)$, we infer from \eqref{Delta Qn represented by Hn in Ij} and \eqref{uniform bound for Qn and Hn in Ij} that 
\begin{equation}\label{L2H2Q}
    \|Q^n\|_{L^2(I_j, H^2(\Omega\setminus \mathcal{S}_j))}\leq C,
\end{equation}
where $C$ is a constant only depending on the initial energy $E(u_0, Q_0)$.

In \eqref{ws:Qn} we substitute a test function $\psi\in C_c^1(I_j\times (\Omega\setminus \mcS_j))$. Then we have $\varphi^n(t)\cdot \psi\equiv 0$ for $t\in I_j$, $n>n_0$ and $\pa_tQ^n$ satisfies 

\begin{equation*}
    \int_{(j-1)\tau}^{j\tau}\int_{\Omega\setminus \mcS_j}\pa_t Q^n:\psi \,dxdt=\int_{(j-1)\tau}^{j\tau}\int_{\Omega\setminus \mcS_j} \left(\Gamma H^n-(u^n\cdot\na )Q^n+(\Sigma^nQ^n-Q^n\Sigma^n)\right):\psi\,dxdt. 
\end{equation*}

Similarly as in \eqref{Qt1} we have,
\begin{equation*}
    \begin{split}
       \left| \int_{\Omega\setminus \mcS_j} H^n:\psi\,dx\right|&\leq \|H^n\|_{L^2(\Omega)}\|\psi\|_{L^2(\Omega\setminus\mcS_j)},\\
       \left| \int_{\Omega\setminus \mcS_j} (u^n\cdot\na) Q^n :\psi \,dx\right|&\leq \|u^n\|_{L^6(\Omega)}\|\na Q^n\|_{L^2(\Omega)}\|\psi\|_{L^3(\Omega\setminus\mcS_j)}\\
        &\leq C\|u^n\|_{H^1(\Om)}\|Q^n\|_{H^1(\Om)}\|\psi\|_{H^1(\Om\setminus\mcS_j)},\\
        \left| \int_{\Omega\setminus \mcS_j} (\Sigma^n Q^n-Q^n\Sigma^n):\psi \,dx\right|&\leq \|u\|_{H^1(\Om)}\|Q^n\|_{L^6(\Om)}\|\psi\|_{L^3(\Om\setminus\mcS_j)}\\
        &\leq C\|u^n\|_{H^1(\Om)}\|Q^n\|_{H^1(\Om)}\|\psi\|_{H^1(\Om\setminus\mcS_j)}.
    \end{split}
\end{equation*}
By \eqref{energy:approx} we know that $\|u^n\|_{L^2(I_j;H^1(\Omega))}+\|Q^n\|_{L^\infty(I_j;H^1(\Omega))}\leq C$, which together with estimates above imply that 
\begin{equation}\label{L2H-1Q}
    \|\pa_t Q^n\|_{L^2(I_j; H^{-1}(\Om\setminus\mcS_j))}\leq C,
\end{equation}
for some constant $C$ depending on $E(u_0, Q_0)$. The estimates \eqref{L2H2Q}, \eqref{L2H-1Q}, together with \eqref{con:Qn} and Lions-Aubin-Simon Lemma (cf. \cite{simon1986compact}) implies that 
\begin{equation*}
    Q^n\rightarrow Q \text{ strongly in }L^2(I_j;H^1(\Omega\setminus \mcS_j)),\quad \forall j=1,..., N_{\sigma(\varepsilon)}.
\end{equation*}

By \eqref{nbh of phi belong to sj}, $(0,T)\times \widetilde{\mathcal{F}}_\sigma(t)\subset \bigcup\limits_{j=1}^{N_\sigma} I_j\times (\Om\setminus\mcS_j)$, thus
\begin{equation*}
    Q^n\rightarrow Q \text{ strongly in }L^2(0,T;H^1(\widetilde{\mathcal{F}}_\sigma(t)).
\end{equation*}
The proof is complete.

\end{proof}
\begin{remark}
A byproduct of the proof of Lemma \ref{strong con lemma:Q} is that $Q^n$ is uniformly bounded in $L^2(0,T; H^2(\widetilde{\mathcal{F}}_\sigma(t)))$ for $n\geq n_0$ and $Q\in L^2(0,T;H^2(\mcF(t)))$ with the norm estimate
\begin{equation*}
    \|Q\|_{L^2(0,T;H^2(\mcF(t)))}\leq C(u_0,Q_0),
\end{equation*}
where the constant is independent of $n,\,\delta$. 
\end{remark}

\begin{lemma}\label{strong con lemma:u}
The sequence $\{u^n\}$ in \eqref{con:un} converges strongly to $u$ in $L^2(0,T; L^2(\Omega))$.
\end{lemma}
\begin{proof}
The strong convergence of $\{\uu\}$ in $L^2(0,T;L^2(\Omega))$ follows largely the proofs of \cite[Theorem 3.2]{MR1870954} and \cite[Proposition 3.1]{F3}. The main difference is that in the equation \eqref{ws:un}, there are several extra terms involving $Q^n, H^n$ which need to be estimated carefully when deriving the norm bound for $\pa_t u^n$ from \eqref{ws:un}. Here we work on $I_j\times \Om$ and let us recall the equation \eqref{ws:un}, where we take $\zeta$ to be compactly supported in time in the interval $I_j$: \begin{multline}\label{3.3} 
-\int\limits_{I_j}\int\limits_{\Omega} \uu\cdot \pa_t\zeta\ dx\ dt = \int\limits_{I_j}\int\limits_{\Omega} \left[(\uu\otimes\uu):\na\zeta-(\mu+n\varphi^n) D(\uu):D(\zeta)-\delta \nabla \Delta u^n\cdot \nabla\Delta \zeta\right]\ dx \ dt  \\
+\int\limits_{I_j}\int\limits_{\Omega} \left[( -(\zeta\cdot\na )Q^n:(H^n+n\varphi^nQ^n) -(\Q H^n-H^n\Q):\na \zeta \right]\ dx \ dt.
\end{multline}
 We need to estimate the terms in the right-hand side of \eqref{3.3} and more precisely we focus on the terms $-(\zeta\cdot\na )Q^n:(H^n+n\varphi^nQ^n)$ and $-(\Q H^n-H^n\Q):\na \zeta$. We substitute $\zeta \in C^1_c(I_j\times \Om)\cap L^2(I_j;K(\mcS_j)\cap H^3(\Om))$ into \eqref{3.3} to obtain
\begin{equation*}
        \varphi^nD(\zeta)=0,\quad \forall\ n>n_0.
\end{equation*}
Let us examine equation \eqref{ws:Hn}. Take test function $\xi=(\zeta\cdot \na)Q^n$ and use $\operatorname{div}\zeta=0$ in $\Omega$, $\zeta\in K(\mcS_j)$ for a.e. $t\in I_j$ to obtain
\begin{equation}\label{div tau term in Q equation}
\begin{split}
    \int_{I_j}\int_{\Om} -(\zeta\cdot \na)Q^n:(H^n+&n\varphi^nQ^n)\,dx\ dt          =\int_{I_j}\int_{\Om} \na Q^n:\na ((\zeta\cdot \na)Q^n)+\f{\pa f_b(Q^n)}{\pa Q}:(\zeta\cdot\na) Q^n\,dx\ dt\\
   &=\int_{I_j}\int_\Om (\zeta\cdot\na)(\f12|\na Q^n|^2+f_b(Q_n))+(\na Q^n\odot \na Q^n):\na\zeta\,dx\ dt\\
   &=\int_{I_j}\int_\Om (\na Q^n\odot \na Q^n):\na\zeta\,dx\ dt.
\end{split}
\end{equation}
 Consequently, by using the relation \eqref{div tau term in Q equation} and the embedding $H^2(\Omega)\hookrightarrow L^{\infty}(\Omega)$, we obtain
\begin{equation*}
   \left|\int_{I_j}\int_\Om  -(\zeta\cdot \na)Q^n:(H^n+n\varphi^nQ^n)\,dx\ dt\right| \leq C\|Q_n\|^{2}_{L^{\infty}(0,T; H^1(\Om))}\|\zeta\|_{L^2(0,T;H^3(\Omega))}.
\end{equation*}
The last term in \eqref{3.3} can be estimated as
\begin{equation*}
    \begin{split}
    &\left| \int_{I_j}\int_\Omega (Q^nH^n-H^nQ^n):\na\zeta \,dx\ dt\right| \leq \int_{I_j} \|Q^n(t)\|_{L^4(\Om)}\|H^n(t)\|_{L^2(\Om)}\|\na\zeta(t)\|_{L^4(\Om)}\ dt \\
    &\leq C\int_{I_j}\|Q^n(t)\|_{H^1(\Om)}\|H^n(t)\|_{L^2(\Om)}\|\zeta(t)\|_{H^2(\Om)}\ dt \leq C\|Q^n\|_{L^{\infty}(0,T; H^1(\Om))}\|H^n\|_{L^2(0,T;L^2(\Om))}\|\zeta\|_{L^2(0,T;H^2(\Om))}.
    \end{split}
\end{equation*}
The other terms in \eqref{3.3} can be estimated directly, similarly as in \cite[Section 7]{MR1870954} and \cite[Proposition 3.1]{F3}. Finally, we obtain
\begin{equation*}
    \left|  \int_{I_j}\int_\Om u^n\pa_t\zeta   \,dxdt \right|\leq C(\delta,E(u_0,Q_0)) \|\zeta\|_{L^2(I_j;H^3(\Om))}.
\end{equation*}
Then Lions-Aubin-Simon Lemma applies and the rest of the proof follows from exactly the same argument as in \cite[Section 7]{MR1870954}. 
\end{proof}


\underline{\textbf{Step 3: Passing to the limit in the H equation.}} In \eqref{ws:Hn}, let us consider a test function\\ $\xi\in L^2(0,T;H_0^1(\widetilde{\mathcal{F}}_\sigma(t)))$ for a small $\sigma$. Using Lemma \ref{lemma 6.1 2d paper} we have, 
\begin{equation*}
    \varphi^n\xi=0,\quad \text{for }n \text{ sufficiently large}.
\end{equation*}
Thus it follows from  \eqref{con:Qn}, \eqref{con:Hn} and Lemma \ref{strong con lemma:Q} that 
\begin{equation}\label{eq:n limit, H}
\int_0^T\int_{\mcF(t)} \left[ \na Q:\na\xi+ \frac{\pa f_b(Q)}{\pa Q}: \xi+H : \xi \right]\ dx\ dt=0.
\end{equation}
Since $\sigma$ can be arbitrarily small, we conclude that \eqref{eq:n limit, H} holds for any $\xi\in L^2(0,T;H_0^1(\mcF(t)))$.

\underline{\textbf{Step 4: Passing to the limit in the Q-tensor equation.}} We recall the equation \eqref{ws:Qn}:
\begin{multline}\label{rc:Qn}
\int\limits_0^T\int\limits_{\Omega} \left[-\Q\cdot \partial_t \psi - (\uu\cdot \nabla)\psi\cdot \Q -\Sigma^n \Q \cdot \psi + \Q\Sigma^n\cdot \psi \right]\ dx \ dt \\ = \int\limits_{\Omega} Q_0 \cdot \psi(0)\ dx +  \int\limits_0^T\int\limits_{\Omega}\Gamma H^n\cdot \psi \ dx \ dt,
\end{multline}

We pick a test function $\psi\in H^1((0,T)\times \Om)\cap L^2(0,T;H_0^1(\mcF(t)))$ such that $\psi(T)=0$. Then from the weak convergence results \eqref{con:un}, \eqref{con:Qn}, \eqref{con:Hn} and the strong convergence results in Lemma \ref{strong con lemma:Q} and Lemma \ref{strong con lemma:u}, we can let $n\ri\infty$ in \eqref{ws:Qn} and get
\begin{equation}\label{eq: n limit, Q}
    \int\limits_0^T\int\limits_{\Om} \left[- Q:\pa_t \psi -\Gamma H : \psi -(u\cdot \nabla)\psi : Q -(\Sigma Q) : \psi + (Q\Sigma): \psi \right]\ dx \ dt =\int_{\Om} Q(0):\psi(0)\,dx
\end{equation}
Note that since $\psi=Q=0$ on $\mcS(\varphi(t))$, \eqref{eq: n limit, Q} can be written as \eqref{ws:Q}.

\underline{\textbf{Step 5: Passing to the limit in the momentum equation.}}

Again we fix a small constant $\sigma$ first. Take the test function $\zeta \in C^1([0,T]\times \Omega)\cap L^2(0,T;K_\sigma(\varphi)\cap H^3(\Om))$ such that $\zeta(T)=0$. Then due to Lemma \ref{lemma 6.1 2d paper}, we have 
\begin{equation*}
    \varphi^n D(\zeta)=0 , \quad \forall\ n\geq n_0.
\end{equation*}
We pass to the limit $n\rightarrow \infty$ in \eqref{ws:un}, the left-hand side will converge to that of \eqref{ws:u}. Now we analyze the behaviour of the right-hand side of \eqref{ws:un} as $n\rightarrow \infty$. We follow a  calculation similar to the one in \eqref{div tau term in Q equation} and use  that $D(\zeta)=0$ in $\mcS_\sigma(\varphi(t))$ to get
\begin{equation*}
    \begin{split}
&\int\limits_0^T\int\limits_{\Omega} \left[( -(\zeta\cdot\na )Q^n:(H^n+n\varphi^nQ^n) -(\Q H^n-H^n\Q):\na \zeta\right]\ dx \ dt\\
       =& \int\limits_0^T\int\limits_{\widetilde{\mathcal{F}}_\sigma(t)}  \left[\nabla Q^n\otimes \nabla Q^n: D(\zeta)+(H^nQ^n-Q^nH^n):\nabla \zeta \right]\ dx \ dt+\int\limits_0^T\int\limits_{\mcS_\sigma(\varphi(t))}  (H^nQ^n-Q^nH^n):\nabla \zeta \ dx \ dt\\
    \end{split}
\end{equation*}
We use   the strong convergence of $Q^n$ (see Lemma \ref{strong con lemma:Q}) and the weak convergence of $H^n$ (see \eqref{con:Hn}) to conclude 
\begin{equation*}
  \begin{split}
&\lim\limits_{n\rightarrow\infty} \int\limits_0^T\int\limits_{\widetilde{\mathcal{F}}_\sigma(t)}  \left[\nabla Q^n\otimes \nabla Q^n: D(\zeta)+(H^nQ^n-Q^nH^n):\nabla \zeta \right]\ dx \ dt\\ = & \int\limits_0^T\int\limits_{\widetilde{\mathcal{F}}_\sigma(t)} \left[\nabla Q\otimes \nabla Q: D(\zeta)+(HQ-QH):\nabla \zeta \right]\ dx \ dt.
\end{split}
\end{equation*}
 Next we want to show
\begin{equation}\label{limit of Hn}
    \lim\limits_{n\rightarrow\infty}\int\limits_0^T \int\limits_{\mcS_\sigma(\varphi(t))}  \left[(H^nQ^n-Q^nH^n):\nabla \zeta \right]\ dx \ dt=\int\limits_0^T\int\limits_{\mcS_\sigma(\varphi(t))}  \left[(HQ-QH):\nabla \zeta \right]\ dx \ dt
\end{equation}
We decompose the integral in the following way:
\begin{equation*}
\begin{split}
    &\int\limits_0^T \int\limits_{\mcS_\sigma(\varphi(t))}  \left[\left((H^nQ^n-Q^nH^n)-(HQ-QH)\right):\nabla \zeta \right]\ dx \ dt\\
   =&\int\limits_0^T \int\limits_{\mcS_\sigma(\varphi(t))\setminus\mcS_{\sigma/k}(\varphi(t))}  ((H^nQ^n-Q^nH^n)-(HQ-QH)):\nabla \zeta \ dx\ dt +\\
   &\qquad  +\int\limits_0^T \int\limits_{ \mcS_{\sigma/k}(\varphi(t))\setminus \mcS(\varphi(t))}  (H^nQ^n-Q^nH^n):\nabla\zeta\,dxdt+\int\limits_0^T \int\limits_{\mcS_{\sigma/k}(\varphi(t))\setminus \mcS(\varphi(t))}  (-HQ+QH):\nabla\zeta\ dxdt\\
&\qquad+\int\limits_0^T \int\limits_{  \mcS(\varphi(t))}  (H^nQ^n-Q^nH^n):\nabla\zeta\,dx\ dt-\int\limits_0^T \int\limits_{  \mcS(\varphi(t))}  (HQ-QH):\nabla\zeta\,dx\ dt \\
   =:& \left(I_1(n,k)+I_2(n,k)+I_3(k)+I_4(n)+I_5\right)
\end{split}
\end{equation*}
Firstly, due to the strong convergence of $Q^n$ in Lemma \ref{strong con lemma:Q}, we obtain
\begin{equation*}
\lim\limits_{n\ri\infty}I_1(n,k)=0\mbox{ for fixed }k.
\end{equation*}
Regarding $I_3(k)$,
\begin{equation*}
    \begin{split}
        |I_3(k)|&\leq C\|H\|_{L^2([0,T]\times\Omega)}\|Q\|_{L^2([0,T]\times\mcS_{\sigma/k}(\varphi(t))\setminus \mcS(\varphi(t)) )}\\
        &\leq C \left(\int_0^T|\mcS_{\sigma/k}(\varphi(t))\setminus \mcS(\varphi(t))|^{\f23} \|Q\|^2_{L^6(\mcS_{\sigma/k}(\varphi(t))\setminus \mcS(\varphi(t)))}dt\right)^{\f12}\\
        &\leq C\left(\int_0^T|\mcS_{\sigma/k}(\varphi(t))\setminus \mcS(\varphi(t))|^{\f23} \|Q\|^2_{L^6(\Omega)}dt\right)^{\f12}\\
        &\leq C\left(\int_0^T|\mcS_{\sigma/k}(\varphi(t))\setminus \mcS(\varphi(t))|^{\f23} \|Q\|^2_{H^1(\Om)}dt\right)^{\f12}\\
        &\leq C\|Q\|_{L^\infty(0,T;H^1(\Om))}\left(\int_0^T |\mcS_{\sigma/k}(\varphi(t))\setminus \mcS(\varphi(t))|^{\f23} \right)^{\f12}\\
        &\rightarrow 0 \ \text{ as }k\ri\infty.
    \end{split}
\end{equation*}
Here we exploit the facts $H\in L^2((0,T)\times\Om)$, $Q\in L^\infty(0,T)\times H^1(\Om))$ and Lemma \ref{strong con lemma:Q}. Furthermore, by the uniform bound of $\|Q^n\|_{ L^\infty((0,T)\times H^1(\Om))}$ and $\|H^n\|_{ L^2((0,T)\times\Om)}$ from \eqref{energy:approx}, we can obtain the convergence of $I_2(n,k)$ following similar calculation. Thus, we have
\begin{equation*}
\lim\limits_{n,k\ri\infty} (I_2(n,k)+I_3(k))=0.
\end{equation*}
Regarding $I_4(n)$, we utilize $\|Q^n\|_{L^2([0,T]\times \mcS(\varphi^n(t))}\ri 0$ deduced from \eqref{energy:approx} to infer that 
\begin{equation*}
    \lim\limits_{n\ri\infty} I_4(n)=0. 
\end{equation*}
Finally, we use that $Q=0$ in $\mcS(\varphi(t))$ to conclude $I_5=0$. Therefore \eqref{limit of Hn} follows from the estimates above.

 Thus by taking the limit $n\rightarrow\infty$ in equation \eqref{ws:un}, we obtain for all $\zeta \in C^1((0,T)\times \Omega)\cap L^2(0,T;K_\sigma(\varphi)\cap H^3(\Omega))$:
 \begin{multline*}
\int\limits_0^T\int\limits_{\Omega} \left[ u\cdot \pa_t\zeta - (u\otimes u): \na\zeta +\mu D(u):D(\zeta)+\delta \na\Delta u\cdot \na\Delta \zeta\right]\ dx \ dt = \int_{\Omega} u(0)\cdot \zeta(0)\,dx\\
+ \int\limits_0^T\int\limits_{\widetilde{\mathcal{F}}_\sigma(t)} \left[(\na Q\odot \na Q):\na \zeta -(QH):\na \zeta+ (HQ):\na\zeta \right]\ dx \ dt+\int\limits_0^T\int\limits_{\mcS_\sigma(\varphi(t))}  \left[(HQ-QH):\nabla \zeta \right]\ dx \ dt.
\end{multline*}
 Now we use a density argument, the definition $\overline{\bigcup\limits_{\sigma>0} K_\sigma(\varphi)}=K_0(\varphi)$ and the fact that $Q=0$ in $\mcS(\varphi(t))$ to conclude that $(u,Q,H)$ solves the equation
\begin{multline}\label{eq: n limit, u}
\int\limits_0^T\int\limits_{\Omega} \left[ u\cdot \pa_t\zeta - (u\otimes u): \na\zeta +\mu D(u):D(\zeta)+\delta \na\Delta u\cdot \na\Delta \zeta\right]\ dx \ dt \\
= \int_{\Omega} u(0)\cdot \zeta(0)\,dx+ \int\limits_0^T\int\limits_{\mathcal{F}(t)} \left[(\na Q\odot \na Q):\na \zeta -(QH):\na \zeta+ (HQ):\na\zeta \right]\ dx \ dt,
\end{multline}
where $\zeta \in H^1((0,T)\times \Omega)\cap L^2(0,T;K_0(\varphi))$.

Finally, after passing to the limit $n\ri\infty$ in \eqref{energy:approx}, we obtain the following energy inequality
\begin{equation}\label{energy: after n limit}
    E(u,Q)(t)+\mu\int_0^t \int_{\Omega}|D u|^2\,dx\,ds 
+ \Gamma\int_0^t\int_{\mathcal{F}(t)} |H|^2\,dx\,ds+\int_0^t\int_\Om \delta|\na\Delta u|^2\,dx\,ds
\le E(u_0,Q_0),
\end{equation}
For any $t\in[0,T)$. Moreover, we can follow the same deduction as in the Step 3 of the proof of Proposition \ref{exist-galerkin} to conclude that we can choose $T=T_{\delta}$ (depending on $\operatorname{dist}(S(\varphi_0),\partial\Omega)$, $\delta$, $u_0$, $Q_0$) such that
\begin{equation*}
\operatorname{dist}(S(\varphi(t)),\partial\Omega)>0 \text{ for }t\in[0,T).
\end{equation*}
Thus, the process of passing $n\ri\infty$ is complete.

\subsection{Pass to the limit as $\delta\rightarrow 0$} 
We first summarise the solution we get from the analysis above. For any $\de>0$, there are solutions on the time interval $(0,T_\delta)$, 
\begin{equation*}
\begin{split}
&u_\de \in L^{\infty}(0,T_\delta; H(\Omega))\cap L^2(0,T_\de;K(\varphi_\de)\cap H^3(\Om)), \\
&Q_\de\in L^{\infty}(0,T_\de; H^1_0(\Om))\cap L^2(0,T_\de;H^2(\Om\setminus \mcS(\varphi_\de(t)))),\ \varphi_\de(t)Q_\de=0,\\
&\varphi_\de\in \mbox{Char}\ (\Omega)\cap C^{0,\frac{1}{p}}(0,T_\de;L^p(\Omega)),\ p\in[1,+\infty), \quad H_\de\in L^2(0,T_\de,L^2(\Om)),
\end{split}
\end{equation*}
that satisfy the following equations 
\begin{equation}\label{ws:Q_de}
\int\limits_0^{T_\de}\int\limits_{\Om} \left[- Q_\de:\pa_t \psi -\Gamma H_\de : \psi +(u_\de\cdot \nabla)Q_\de : \psi -(\Sigma_\de Q_\de) : \psi + (Q_\de\Sigma_\de): \psi \right]\ dx \ dt =\int_{\Om} Q(0):\psi(0)\,dx
\end{equation}
\begin{equation}\label{ws:H_de}
\int_0^{T_\de}\int_{\Om} \left[ \na Q_\de:\na \xi+ \frac{\pa f_b(Q_\de)}{\pa Q}: \xi+H_\de : \xi \right]\ dx\ dt=0
\end{equation}
\begin{multline}\label{ws:u_de}
\int\limits_0^{T_\de}\int\limits_{\Omega} \left[-u_\de\cdot \pa_t\zeta - (u_\de\otimes u_\de): \na\zeta +\mu D(u_\de):D(\zeta)+\delta \na\Delta u_\de :\na\Delta \zeta\right]\ dx \ dt \\
= \int_{\Omega} u(0)\cdot \zeta(0)\,dx+ \int\limits_0^{T_\de}\int\limits_{\Om} \left[(\na Q_\de\odot \na Q_\de):\na \zeta -(Q_\de H_\de):\na \zeta+ (H_\de Q_\de):\na\zeta \right]\ dx \ dt,
\end{multline}
\begin{equation}\label{ws:solid_de}
\int\limits_0^{T_\de}\int\limits_{\Omega} \varphi_\de \left[\partial_t\eta+(u_\de\cdot\nabla)\eta\right]\ dx\ dt= - \int\limits_{\Omega} \varphi(0) \eta(0)\ dx,
\end{equation}
 for any functions 
 \begin{align*}
 &\psi\in H^1((0,T_\de)\times \Om)\cap L^2(0,T_\de; H_0^1(\Om\setminus \mcS(\varphi_\de(t)))),\  \psi(T_\de)=0,\\
 &\xi\in L^2(0,T_\de;H^1_0(\Om\setminus \mcS(\varphi_\de(t)))), \\
 &\zeta\in H^1((0,T_\de)\times \Omega)\cap L^2(0,T;K(\varphi_\de)), \ \zeta(T_\de)=0,\\
 &\eta\in C^1([0,T_\de]\times\Omega), \ \eta(T_\de)=0.   
 \end{align*}.

Here we always integrate over the whole domain $\Om$ since the solid part $\mcS(\varphi_\de(t))$ is varying with respect to $\de$. Note that here $T_\de$ satisfies that $\mbox{dist}( \mcS(\varphi_\de(t)), \pa\Om)>0$ for any $t\in (0,T_\de)$. Moreover, according to the analysis in Step 3 of the proof of Proposition \ref{exist-galerkin}, $T_\de$ depends on $\de$ and the initial data. In this subsection we will utilize the property of the rigid velocity field to show that there exists a time $T(u_0,Q_0,\varphi_0)$ such that $T(u_0,Q_0,\varphi_0)<T_\de$ for any small $\de$ uniformly. The arguments follow largely \cite[Section 3.2]{HOST}.

Let $X_{s,t}^\delta(y)$ be the unique solution to 
\begin{equation*}
    \f{d X_{s,t}^\de(y)}{dt}=u_\de(t,X_{s,t}^\de(y)),\quad X_{s,s}^\de(y)=y,\quad \forall y\in \Omega, (s,t)\in[0,T_\de]^2.
\end{equation*}
Then $\varphi_\de$ satisfies
\begin{equation*}
    \varphi_\de(x,t)=\varphi_0(X_{0,t}^\de(x)),\quad \mcS(\varphi_\de(t))=X_{0,t}^\de(\mcS(\varphi_0)).
\end{equation*}
Since $u_\de\in K(\varphi_\de)$ for a.e. $t\in[0,T_\de]$, there exist functions $h_\de(t)$ and $\omega_\de(t)$ such that 
\begin{equation*}
    u_\de(x,t)=\f{d\, h_\de(t)}{d\,t}+\omega_\de(t)\times (x-h_\de(t)),\quad \forall x\in\mcS(\varphi_\de(t)),\ t\in[0,T_\de],
\end{equation*}
where $h_\de(t)$ and $\omega_\de(t)$ are the centre of mass and angular velocity of $\mcS(\varphi_\de(t))$ respectively. The definition of the centre of mass, mass of the body \eqref{cm} and moment of inertia \eqref{moment of inertia} give:
\begin{equation*}
\int\limits_{\mcS(\varphi_\delta(t))} |u_\delta(x,t)|^2\,dx=m\left|\f{d\, h_\de(t)}{d\,t}\right|^2+ J\omega_{\delta}\cdot \omega_{\delta}  
\end{equation*}
It follows immediately from the relation above and $\|u_\de\|_{L^\infty(0,T_\de; L^2(\Omega))}\leq C(u_0,Q_0,\varphi_0)$ that 
\begin{equation*}
    \left\|\f{d\,h_\de(t)}{d\,t}\right\|_{L^\infty(0,T_\de)}\leq C(u_0,Q_0,\varphi_0)\quad\mbox{ and }\quad\|\omega_\delta(t)\|_{L^\infty(0,T_\de)}\leq C(u_0,Q_0,\varphi_0).
\end{equation*}
Therefore we have 
\begin{equation*}
    \|u_{\delta}\|_{L^\infty(0,T_\de;\mcS(\varphi_\de))}\leq C(u_0,Q_0,\varphi_0).
\end{equation*}
Since $X_{s,t}^\de(y)$ satisfies
\begin{equation*}
    X_{s,t}^\de(y)-y=\int_s^t u_\de( X_{s,\tau}^\de(y),\tau)\,d\tau, \quad \forall y\in \mcS(\varphi_\de(s)),
\end{equation*}
we can infer that $X_{s,t}^\de(y)$ restricted to $\mcS(\varphi_\de(s))$ is Lipschitz continuous in $t$  uniformly with respect to $\delta$. Moreover, it follows by the same deduction as in the Step 3 of the proof of Proposition \ref{exist-galerkin} (but now with time $T_0$ uniform with respect to $\delta$) that there exists a $T_0=T_0(u_0,Q_0,\varphi_0)$ and an $\alpha>0$ such that for any $\delta>0$ and $t<T_0$,
\begin{equation}\label{dis to bdy}
    \mbox{dist}(\mcS(\varphi_\de(t)),\pa\Om)\geq \alpha,
\end{equation}
which further implies that $T_\de>T_0$ for all $\delta>0$. 

Define the rigid velocity field $U_\de(t,x): [0,T_0]\times\mathbb{R}^3\rightarrow \mathbb{R}^3$ and the corresponding isometries $M_{s,t}^\de(y)$ by
\begin{equation*}
\begin{split}
    &\quad U_\de(t,x)=\f{d\, h_\de(t)}{d\, t}+\omega_\de(t)\times (x-h_\de(t)),\quad \forall \ x\in \mathbb{R}^3,\ t\in[0,T_0],\\
    & \f{d\, M_{s,t}^\delta(y) }{d\,t}=U_\de(t, M_{s,t}^\delta(y)),\quad M_{s,s}^\de(y)=y, \quad \forall y\ \in\mathbb{R}^3,\ (s,t)\in [0,T_0]^2. 
\end{split}
\end{equation*}
Utilizing the uniform bounds of $ \|\f{d\,h_\de(t)}{d\,t}\|_{L^\infty(0,T_\de)}$ and $ \|\omega_\delta(t)\|_{L^\infty(0,T_\de)}$ we have that $U_\delta(t,x)$ is bounded on $[0,T_0]\times K$ uniformly with respect to $\delta$ for any $K\subset\joinrel\subset \mathbb{R}^3$. Consequently,  $M_{s,t}^\de(y)$ is Lipschitz in $t,\,s$ uniformly with respect to $\de$. By the same argument as in the Step 1 of Subsection \ref{subsec: n limit}, we obtain the following lemma:
\begin{lemma}\label{lemma:isometry}
There exist $\varphi(t,x)$ and $M_{s,t}(y)$ such that, up to some subsequence $\de\ri 0$,
\begin{align*}
&\varphi_\de \ri \varphi \text{ strongly in } L^2([0,T_0]\times \Omega),\\
&M_{s,t}^\de \ri M_{s,t} \text{ strongly in }C([0,T_0]^2\times\mathbb{R}^3),\\
& \mcS(\varphi(t))=M_{0,t}(\mcS(\varphi_0)),\quad \forall t\in[0,T_0].
\end{align*}
The isometry $M_{0,t}$ is Lipschitz continuous with respect to $t\in [0,T_0]$, i.e. 
\begin{equation*}
    \|M_{0,t_1}-M_{0,t_2}\|_{C(K,\mathbb{R}^3)}\leq c(K,u_0,Q_0,\varphi_0)|t_1-t_2|,\ \ \forall K\subset\joinrel\subset \mathbb{R}^3,\; t_1,t_2\in [0,T_0].
\end{equation*}
Moreover, for any $\sigma>0$, there exists $\de_0$ (depending on $\sigma$ and initial data) such that
\begin{equation}\label{rel:Sdelta}
\mcS(\varphi_\de(t))\subset \mcS_\sigma(\varphi(t))\mbox{ and }\mcS(\varphi(t))\subset \mcS_\sigma(\varphi_\de(t))\mbox{ hold for all }\de<\de_0\mbox{ and }t\in[0,T_0].
\end{equation}
\end{lemma}

Now for the system \eqref{ws:Q_de}--\eqref{ws:solid_de} we can pass to the limit $\de\rightarrow 0$ in the same way as in Subsection \ref{subsec: n limit} to get a weak solution $(u,Q,\varphi,H)$ to \eqref{Q:fluid}--\eqref{initial} in the sense of Definition \ref{def:weak sol}.  The energy inequality \eqref{energy:main} follows by taking $\de\ri 0$ in \eqref{energy: after n limit}. The distance to the boundary has been characterised by $\alpha$, as in \eqref{dis to bdy}. Now we can choose $\sigma=\frac{\alpha}{4}$ in \eqref{rel:Sdelta} to obtain 
\begin{equation*}
\mcS(\varphi_\de(t))\subset \mcS_{\frac{\alpha}{4}}(\varphi(t))\subset \mcS_{\frac{\alpha}{2}}(\varphi_\de(t))\mbox{ holds for all }\de<\de_0\mbox{ and }t\in[0,T_0],
\end{equation*}
which implies 
\begin{equation*}
    \mbox{dist}(\mcS(\varphi(t)),\pa\Om)\geq \frac{\alpha}{2} >0.
\end{equation*}
This completes the proof of the following local existence theorem.
\begin{theorem}[Local existence]\label{local existence}
Assume $u_0\in H(\Om)$, $Q_0\in H_0^1(\mcF(0);M)$, $\mcS_0=\mcS(\varphi_0)$, $\mbox{dist}(\mcS_0,\pa\Om)=2\alpha>0$, and the boundaries $\pa\mcS,\,\pa\Om$ are of class $C^2$. There exists $T_0=T_0(u_0,Q_0,\varphi_0,\alpha)$ such that the system \eqref{Q:fluid}--\eqref{initial} admits at least one weak solution $(u,Q,H,\varphi)$ on the time interval $(0,T_0)$ which satisfies the energy inequality \eqref{energy:main}. In particular the solution satisfies
\begin{equation*}
    \mbox{dist}(\mcS(\varphi(t)),\pa\Omega)\geq \alpha,\quad \forall\ t\in (0,T_0).
\end{equation*}
\end{theorem}
\subsection{Extend the solution beyond the collision time}

The last step of the proof of Theorem \ref{exist:main} is to extend the local solution obtained in Theorem \ref{local existence} to any time interval $[0,T]$. The argument essentially follows \cite[Section 4]{F3}. 

\begin{lemma}\label{extension of sol}
Let $(u,Q,H,\varphi)$ be a weak solution in the sense of Definition \ref{def:weak sol} on $(0,T_0)$ such that \begin{equation*}
    \mbox{dist}(\mcS(\varphi(t)),\pa\Om)>0,
\end{equation*}
for all $t\in(0,T_0)$.

Then there exist $u(T_0)\in H(\Om)$, $Q(T_0)\in H_0^1(\mcF(T_0))$, $\varphi(T_0)\in \mbox{Char}(\Om)$ such that 
\begin{align}
   \label{exist u(T0)} &\int\limits_\Omega \zeta \cdot u(T_0)\,dx=\lim\limits_{t\uparrow T_0}\int\limits_\Om\zeta \cdot u(t)\,dx\ \text{ for any }\zeta\in K_0(\varphi(T_0)),\\
   \label{exist Q(T_0)} 
   &\langle\psi, \tilde{Q}(T_0)\rangle
   =\lim\limits_{t\uparrow T_0} \langle\psi, \tilde{Q}(t)\rangle
   \text{ for any }\psi\in H^{-1}(\Omega),\\
    &\int\limits_\Omega \eta  \varphi(T_0)\,dx=\lim\limits_{t\uparrow T_0}\int\limits_\Om\eta \varphi(t)\,dx\ \text{ for any }\eta\in C^1(\Om),
\end{align}
where $\mcF(T_0)=\Om\setminus \mcS(\varphi(T_0))$, $\langle \cdot,\cdot  \rangle$ denotes the usual $L^2$ inner product between $H^{-1}(\Om)$ and $H^1_0(\Om)$, and $\tilde{Q}(t)$ is defined by 
\begin{equation*}
    \tilde{Q}(t,x)=\begin{cases} Q(t,x), & x\in\mcF(t),\\
    0,& x\in\mcS(t),
    \end{cases},\quad \forall t\in(0,T_0].
\end{equation*}

Finally, if $(u, Q, H, \varphi)$ weakly solves the system on $(T_0,T_1)\times \Om$ with the initial data $(u(T_0), Q(T_0), \varphi(T_0))$, then $(u,Q,H,\varphi)$ is a weak solution on $(0,T_1)\times \Om$.
\end{lemma}

\begin{proof}
First we establish the existence of $\varphi(T_0)$. According to Lemma \ref{lemma:isometry}, $M_{0,t}$ is Lipschitz continuous with respect to $t$. By the same argument as in \cite[Lemma 2.2]{F3}, we can define $\varphi(T_0)$ by
\begin{equation*}
    \mcS(\varphi(T_0)):=M_{0,T_0}(\mcS(\varphi_0))=\lim\limits_{t\uparrow T_0}M_{0,t}(\mcS(\varphi_0)). 
\end{equation*}
Immediately we set
\begin{equation*}
    \mcF(T_0):=\Om\setminus \mcS(\varphi(T_0))=\lim\limits_{t\uparrow T_0}F(t).
\end{equation*}
Fix $\sigma\ll 1$ and let $\zeta_0$ be an arbitrary test function in $K_\sigma(\varphi(T_0))$. Because of the Lipschitz continuity of $M_{0,t}$, there exists $t_\sigma$ independent of $T_0$ such that 
\begin{equation*}
    \mcS(\varphi(t))\subset \mcS(\varphi(T_0))_\sigma,\quad \forall t\in [T_0-t_\sigma,T_0].
\end{equation*}
Take $\Phi(t) \in \mathcal{D}([T_0-t_\sigma,T_0)$. We substitute $\zeta=\zeta_0\Phi(t)$ in \eqref{ws:u} to obtain
\begin{multline}\label{ws:u T0-t_s to T0}
\int\limits_{T_0-t_\sigma}^{T_0}\int\limits_{\Omega} \left[-(u\cdot\zeta_0)\cdot \pa_t\Phi - (u\otimes u): \na\zeta_0\Phi +\mu D(u):D(\zeta_0)\Phi\right]\ dx \ dt \\
=  \int\limits_{T_0-t_\sigma}^{T_0}\int\limits_{\mcF(t)} \left[(\na Q\odot \na Q):\na \zeta_0\Phi -(QH-HQ):\na \zeta_0\Phi \right]\ dx \ dt.
\end{multline}
Thus we find the following equality holds in the distribution sense on $(T_0-t_\sigma,T_0)$:
\begin{align*}
    \f{d}{dt} \bigg(\int\limits_{\Om} u(t)\cdot \zeta_0\,dx\bigg)=\int\limits_\Om \bigg( (u&\otimes u):\na \zeta_0-\mu D(u):D(\zeta_0) \bigg) \,dx\\
    &+\int\limits_{\mcF(t)}(\na Q\odot\na Q-QH+HQ):\na\zeta_0 \,dx,
\end{align*}
which further implies 
\begin{equation*}
    \lim\limits_{t\uparrow T_0}\int_{\Omega} u(t)\cdot \zeta_0\,dx \ \text{ has a limit }
\end{equation*}
for any test function $\zeta_0\in K_\sigma(\varphi(T_0))$. By a density argument we can identify this limit as a linear functional on the vector space $K_0(\varphi(T_0))$, i.e. 
\begin{equation*}
    \lim\limits_{t\uparrow T_0}\int_{\Omega} u(t)\cdot \zeta\,dx=L\zeta\quad \text{ for any }\zeta\in K_0(\varphi(T_0)). 
\end{equation*}
Moreover, by the energy inequality we have 
\begin{equation*}
    |L\zeta|\leq \|\zeta\|_{L^2(\Om)} \liminf\limits_{t\uparrow T_0}\|u(t)\|_{L^2}\leq \|\zeta\|_{L^2(\Om)}\sqrt{2E(u_0,Q_0)}.
\end{equation*}
By Hahn-Banach theorem $L$ can be extended as a bounded linear functional on $H(\Omega)$. Then utilizing the Riesz representation theorem we find a function $u(T_0)\in H(\Om)$ such that \eqref{exist u(T0)} holds and
\begin{equation*}
    \|u(T_0)\|_{L^2(\Om)}\leq \liminf\limits_{t\uparrow T_0} \|u(t)\|_{L^2(\Om)}.
\end{equation*}

By a similar argument we can also find a function $Q(T_0)\in H_0^1(\mcF(T_0))$ that satisfies \eqref{exist Q(T_0)}. 

For the last part of the lemma, if $(u, Q, H, \varphi)$ is a weak solution of the system on $(T_0,T_1)\times \Om$ with the initial data $(u(T_0), Q(T_0), \varphi(T_0))$, then $(u,Q,H,\varphi)$ is a weak solution on the combined time interval $(0,T_1)$ which directly follows from the definition of the weak solution. We omit the details here.
\end{proof}

An immediate consequence of the lemma above is that the weak solution in Theorem \ref{local existence} can exist upon the collision time of the colloid and boundary. Otherwise one can always extend the weak solution to a larger time interval as long as $\mbox{dist}(\mcS(t),\pa\Om)>0$. We have the following corollary.

\begin{cor}
Under the hypotheses of Theorem \ref{local existence}, the system \eqref{Q:fluid}--\eqref{initial} admits a weak solution $(u,Q,H,\varphi)$ on a maximal time interval $(0,T_{max})$. If $T_{max}<\infty$, then $\mbox{dist}(\mcS(\varphi(T_{max})),\pa\Om)=0$.
\end{cor}

\begin{proof}[Conclusion of the proof of Theorem \ref{exist:main}]
It suffices to find a weak solution on any time interval $[0,T]$. Firstly, if $T\leq T_{max}$ then there is nothing to prove. Assume $T>T_{max}$. Set 
\begin{equation*}
    \Om_1:=\Om\setminus \overline{\mcS}(\varphi(T_0)).
\end{equation*}
As in \cite{F3}, we adopt the simplest contact condition that the rigid body will stick to the boundary after the collision. At time $T_{max}$, by Lemma \ref{extension of sol} one can define the functions $u(T_{max})\in H(\Om)$, $Q(T_{max})\in H_0^1(\mcF(T_{max}))$. Then on $[T_{max},T]\times \Om_1$, we seek for the solution $(u_1,Q_1)$ that solves the following system
\begin{align*}
    &u_1\in L^\infty(T_{max},T;H(\Om_1))\cap L^2(T_{max},T; V(\Om_1)),\quad Q_1\in L^\infty(T_{max},T;H_0^1(\Om_1))\cap L^2(T_{max},T;H^2(\Om_1));\\
    &\int\limits_{T_{max}}^T\int\limits_{\Om_1} \left[- Q_1:\pa_t \psi +(-\Gamma H_1+(u_1\cdot \nabla)Q_1-\Sigma_1 Q_1+Q_1\Sigma_1) : \psi   \right]\ dx \ dt =\int_{\Om_1} Q_1(T_{max}):\psi(T_{max})\,dx,\\
    &H_1=\Delta Q_1-\frac{\pa f_b(Q_1)}{\pa Q},\\
    &\int\limits_{T_{max}}^T\int\limits_{\Omega_1} \left[-u_1\cdot \pa_t\zeta - (u_1\otimes u_1): \na\zeta +\mu D(u_1):D(\zeta)\right]\ dx \ dt \\
&\qquad\qquad\quad = \int_{\Omega_1} u(T_{max})\cdot \zeta(T_{max})\,dx+ \int\limits_{T_{max}}^T\int\limits_{\Om_1} \left[(\na Q_1\odot \na Q_1):\na \zeta -(Q_1H_1-H_1Q_1):\na \zeta \right]\ dx \ dt,
\end{align*}
for any test functions $\psi\in H^1([T_{max},T]\times\Om_1)\cap L^2(T_{max},T; H_0^1(\Om_1))$ and $\zeta\in H^1([T,T_{max}]\times\Om_1)\cap L^2(T_{max},T; V(\Om_1))$ satisfying $\psi(T)=0$ and $\zeta(T)=0$. The system above describes the weak system of the Beris-Edwards model for nematic liquid crystals on a bounded domain with homogeneous Dirichlet boundary condition. The existence of such a weak solution $(u_1,Q_1)$ is established in \cite{abels2014well,MR3274285} by the Galerkin method.

Finally we define the functions $(\tilde{u}, \tilde{Q},\tilde{H},\tilde{\varphi})$ on $[0,T]\times \Om$ by
\begin{equation*}
        (\tilde{u}, \tilde{Q},\tilde{H},\tilde{\varphi}) =\begin{cases}
        (u,Q,H,\varphi), & t\in [0,T_{max}],\\
        (u_1\chi(\Om_1),Q_1\chi(\Om_1),H_1\chi(\Om_1), (1-\chi(\Om_1))), & t\in(T_{max},T].
        \end{cases}
\end{equation*}

Then it is straightforward to verify that
$ (\tilde{u}, \tilde{Q},\tilde{H},\tilde{\varphi})$ is a weak solution of the system \eqref{Q:fluid}--\eqref{initial} in the sense of Definition \ref{def:weak sol}, which ends the proof of the theorem.
\end{proof}

\begin{remark}\label{rmk:multi-particle}
    Theorem \ref{exist:main} can be easily extended to the case of multiple colloids. Suppose the liquid crystal flow contains a family of $k$ rigid bodies, the positions of which are represented by $\mcS^i(t)$, $i=1,...,k$. The function $\varphi(t)$ becomes the characteristic function of $\bigcup\limits_{i=1}^k\mcS^i(t)$. The weak solution can be defined in the same way as Definition \ref{def:weak sol}.
    Set
\begin{equation*}
    d(t):=\min\{\min\limits_{i,j}\mbox{dist}(\mcS^i(t),\mcS^j(t)),\min\limits_i \mbox{dist}(\mcS^i(t),\pa\Om)\},
\end{equation*}
and assume $d(0)>0$. Then following the same proof of Theorem \ref{local existence}, we can find a weak solution up to the time $T_{max}$ such that $d(T_{max})=0$. At the collision time $T_{max}$, let $J_0$ denote the set of index $i$ such that $\mcS^i$ touches the boundary, i.e.   
\begin{equation*}
    J_0:=\{i\in [1,...,k]: \pa\mcS^i(T_{max})\cap\pa\Omega\neq \varnothing\}.
\end{equation*}
Set 
$$
\Om_1=\Om\setminus (\bigcup\limits_{i\in J_0} \overline{\mcS^i}(T_{max})).
$$

 Moreover for any $i,j\in [1,...,k]\setminus J_0$, if $\mbox{dist}(\mcS^i(T_{max}), \mcS^j(T_{max}))=0$, then these two colloids merge into a new colloid $\tilde{\mcS}^i$. Repeat this merging process until the pairwise distance of all the colloids at time $T_{max}$ are positive. Then the flow can be extended to a larger time $T_{max}^2$ which represents the next collision time. Consequently, the existence of a weak solution can be established on any time interval $[0,T]$. 
\end{remark}

\begin{remark}
We want to give a particular example where the solid will not touch the boundary. We consider the system \eqref{Q:fluid}--\eqref{initial} with \eqref{linear momentum:body} replaced by
\begin{equation*}
mh''= -\int\limits_{\partial \mathcal{S}(t)} (2\mu D+\tau+\sigma-p\mathbb{I}_3) n\, d\Gamma + w(t) , 
\end{equation*}
where $w(t)$ is a feedback law of the form
\begin{equation}\label{feedback}
w(t)= k_p(h_1-h(t))-k_d h'(t).
\end{equation}
In control engineering, this feedback \eqref{feedback} is known as a proportional-derivative (PD) controller. The feedback $w(t)$ is generated by a spring (with spring constant $k_p>0$) and a mechanical damper (with constant $k_d\geq 0$) connected between $h(t)$ (centre of mass of $\mathcal{S}(t)$) and a fixed point $h_1\in\Omega$. This kind of feedback in the context of ``motion of rigid body in the fluid'' has been considered in \cite{MR3397328} (Incompressible fluid), \cite{MR4238201} (Compressible fluid). We can follow the calculations of energy inequality $($\cref{galerkin:energy}$)$ and \cite[Corollary 4.3]{MR3397328} to obtain: let $h_1\in\Omega$ such that $\operatorname{dist}(h_1,\partial\Omega)\geq \alpha>0$ and $w$ satisfies feedback law \eqref{feedback}. Then
\begin{multline*}
\frac{1}{2}\int_{\mcF(t)}|u|^2\,dx+m|h'|^2+J\omega\cdot\omega+\int_{\mcF(t)}\left(\frac{1}{2}|\nabla Q|^2+ f_{b}(Q)\right)\,dx +k_p|h(t)-h_1|^2\\ \leq \frac{1}{2}\int_{\mcF_0}|u_0|^2\,dx+m|\ell_0|^2+J|\omega_0|^2+\int_{\mcF_0}\left(\frac{1}{2}|\nabla Q_0|^2+ f_{b}(Q_0)\right)\,dx +k_p|h(0)-h_1|^2.
\end{multline*}
Moreover, there exists $\delta$, $\varepsilon$ such that if 
\begin{equation*}
\frac{1}{2}\int_{\mcF_0}|u_0|^2\,dx+m|\ell_0|^2+J|\omega_0|^2+\int_{\mcF_0}\left(\frac{1}{2}|\nabla Q_0|^2+ f_{b}(Q_0)\right)\,dx +k_p|h(0)-h_1|^2 \leq \delta^2,
\end{equation*}
then 
\begin{equation*}
\operatorname{dist}(h(t),\partial\Omega) \geq \varepsilon > 0.
\end{equation*}
\end{remark}

\section{Appendix}\label{appendix}

 It is clear that we can derive the weak formulations \eqref{ws:Q}--\eqref{ws:H} from the equations \eqref{Q:fluid} and \eqref{newH}. We want to explain here how to obtain the weak formulation \eqref{ws:u} of the momentum equation with the help of fluid equations \eqref{vel:fluid}--\eqref{div:fluid} and the rigid body equations \eqref{linear momentum:body}--\eqref{angular momentum:body}.

{\bf Derivation of the weak formulation \eqref{ws:u}:}  Let us consider the test function $\zeta\in H^1((0,T)\times \Omega))\cap L^2(0,T;K_0(\varphi)), \ \zeta(T)=0$. We multiply the equation \eqref{vel:fluid} by $\zeta$:
\begin{align}\label{eq1}
\int\limits_0^T\int\limits_{\mcF(t)}\left({\partial_t u^{\mcF}} + \left(u^{\mcF}\cdot \nabla\right)u^{\mcF} + \nabla p^{\mcF} -\mu\Delta u^{\mcF} - \operatorname{div}(\tau +\sigma)\right)\cdot\zeta\ dx \ dt= 0.
\end{align}
Observe that by using the Reynolds transport theorem, we obtain
\begin{equation}\label{Reytransport}
\frac{d}{dt}\left(\int\limits_{\mcF(t)}u^{\mcF}\cdot\zeta\right) = \int\limits_{\mcF(t)} \frac{\partial}{\partial t} u^{\mcF}\cdot\zeta+ \int\limits_{\mcF(t)} u^{\mcF}\cdot\frac{\partial}{\partial t}\zeta + \int\limits_{\partial\mcF(t)} (u^{\mcF}\cdot\zeta)(u^{\mcF}\cdot n).
\end{equation}
Using the divergence free condition \eqref{div:fluid} and integrating by parts, we have
\begin{multline}\label{eq2}
\int\limits_{\mcF(t)}\left( \left(u^{\mcF}\cdot \nabla\right)u^{\mcF} + \nabla p^{\mcF} -\mu\Delta u^{\mcF} - \operatorname{div}(\tau +\sigma)\right)\cdot\zeta \ dx  \\ = \int\limits_{\mcF(t)} \left(- (u^{\mcF}\otimes u^{\mcF}) +\mu D(u^{\mcF})+ (\tau +\sigma)\right): \na\zeta \ dx  -\int\limits_{\partial \mathcal{S}(t)} (2\mu D(u^{\mcF})+\tau+\sigma-p^{\mcF}\mathbb{I}_3) n\cdot\zeta\, d\Gamma \\ + \int\limits_{\partial \mcS(t)} (u^{\mcF}\cdot\zeta)(u^{\mcF}\cdot n)\ d\Gamma.
\end{multline}
Moreover, as $\operatorname{div}\zeta=0$ in $\Omega$ and $D(\zeta)=0$ in $\mcS(t)$, there exists $\ell_{\zeta}$, $\omega_{\zeta}$ such that $\zeta (x,t)=\ell_{\zeta}(t)+ \omega_{\zeta}(t)\times (x-h(t))$ for $(t,x)\in (0,T)\times\mcS(t)$. We multiply the equations \eqref{linear momentum:body} and \eqref{angular momentum:body} by $\ell_{\zeta}$ and $\omega_{\zeta}$ respectively to obtain
\begin{equation}\label{eq3}
-\int\limits_0^T\int\limits_{\partial \mathcal{S}(t)} (2\mu D(u^{\mcF})+\tau+\sigma-p^{\mcF}\mathbb{I}_3) n\cdot\zeta\, d\Gamma\ dt = \int\limits_0^T (m\ell'\cdot\ell_{\zeta} + J\omega'\cdot\omega_{\zeta})\ dt + m\ell(0)\cdot\ell_{\zeta}(0) + J\omega(0)\cdot\omega_{\zeta}(0).
\end{equation}
We use the definition \eqref{def:u} of uniform velocity field $u$ along with the identities \eqref{Reytransport}, \eqref{eq2}, \eqref{eq3} in the equation \eqref{eq1} to obtain the weak formulation \eqref{ws:u}:
\begin{multline*}
\int\limits_0^T\int\limits_{\Omega} \left[-u\cdot \pa_t\zeta - (u\otimes u): \na\zeta +\mu D(u):D(\zeta)\right]\ dx \ dt 
= \int_{\Omega} u(0)\cdot \zeta(0)\,dx+ \int\limits_0^T\int\limits_{\mcF(t)} \left[-\tau:\na \zeta -\sigma:\na\zeta \right]\ dx \ dt.
\end{multline*}


\bibliography{LC_Dirichlet}

\begin{thebibliography}{10}

\bibitem{abels2014well}
{\sc H.~Abels, G.~Dolzmann, and Y.~Liu}, {\em Well-posedness of a fully coupled
  navier--stokes/q-tensor system with inhomogeneous boundary data}, SIAM
  Journal on Mathematical Analysis, 46 (2014), pp.~3050--3077.

\bibitem{abels2016strong}
{\sc H.~Abels, G.~Dolzmann, and Y.~Liu}, {\em Strong solutions for the
  {B}eris-{E}dwards model for nematic liquid crystals with homogeneous
  {D}irichlet boundary conditions}, Adv. Differential Equations, 21 (2016),
  pp.~109--152.

\bibitem{cavaterra2016global}
{\sc C.~Cavaterra, E.~Rocca, H.~Wu, and X.~Xu}, {\em Global strong solutions of
  the full navier--stokes and q-tensor system for nematic liquid crystal flows
  in two dimensions}, SIAM Journal on Mathematical Analysis, 48 (2016),
  pp.~1368--1399.

\bibitem{CST}
{\sc C.~Conca, J.~S. Martin, and M.~Tucsnak}, {\em Existence of solutions for
  the equations modelling the motion of a rigid body in a viscous fluid},
  Commun. Partial Differential Equation, 25 (2000), pp.~1019--1042.

\bibitem{dai2016asymptotic}
{\sc M.~Dai, E.~Feireisl, E.~Rocca, G.~Schimperna, and M.~E. Schonbek}, {\em On
  asymptotic isotropy for a hydrodynamic model of liquid crystals}, Asymptotic
  Analysis, 97 (2016), pp.~189--210.

\bibitem{DEES1}
{\sc B.~Desjardins and M.~J. Esteban}, {\em Existence of weak solutions for the
  motion of rigid bodies in a viscous fluid}, Arch. Rational Mech. Anal., 146
  (1999), pp.~59--71.

\bibitem{DiPerna1989}
{\sc R.~DiPerna and P.~Lions}, {\em Ordinary differential equations, transport
  theory and sobolev spaces.}, Inventiones mathematicae, 98 (1989),
  pp.~511--548.

\bibitem{du2020suitable}
{\sc H.~Du, X.~Hu, and C.~Wang}, {\em Suitable weak solutions for the
  co-rotational beris--edwards system in dimension three}, Archive for Rational
  Mechanics and Analysis, 238 (2020), pp.~749--803.

\bibitem{F3}
{\sc E.~Feireisl}, {\em On the motion of rigid bodies in a viscous
  incompressible fluid}, Journal of Evolution Equations, 3 (2003),
  pp.~419--441.

\bibitem{guillen2014weak}
{\sc F.~Guill\'{e}n-Gonz\'{a}lez and M.~A. Rodr\'{\i}guez-Bellido}, {\em Weak
  time regularity and uniqueness for a q-tensor model}, SIAM Journal on
  Mathematical Analysis, 46 (2014), pp.~3540--3567.

\bibitem{MR3274285}
\leavevmode\vrule height 2pt depth -1.6pt width 23pt, {\em Weak solutions for
  an initial-boundary {$Q$}-tensor problem related to liquid crystals},
  Nonlinear Anal., 112 (2015), pp.~84--104.

\bibitem{GLSE}
{\sc M.~D. Gunzburger, H.~C. Lee, and A.~Seregin}, {\em Global existence of
  weak solutions for viscous incompressible flow around a moving rigid body in
  three dimensions}, J. Math. Fluid Mech., 2 (2000), pp.~219--266.

\bibitem{HOST}
{\sc K.-H. Hoffmann and V.~N. Starovoitov}, {\em On a motion of a solid body in
  a viscous fluid. {T}wo-dimensional case}, Adv. Math. Sci. Appl., 9 (1999),
  pp.~633--648.

\bibitem{huang2015global}
{\sc J.~Huang and S.~Ding}, {\em Global well-posedness for the dynamical
  q-tensor model of liquid crystals}, Science China Mathematics, 58 (2015),
  pp.~1349--1366.

\bibitem{lin2014recent}
{\sc F.~Lin and C.~Wang}, {\em Recent developments of analysis for hydrodynamic
  flow of nematic liquid crystals}, Philosophical Transactions of the Royal
  Society A: Mathematical, Physical and Engineering Sciences, 372 (2014),
  p.~20130361.

\bibitem{Lions1996}
{\sc P.-L. Lions}, {\em Mathematical topics in fluid mechanics,Vol 1}, Oxford
  Lecture Series in Mathematics and its Applications, The Oxford University
  Press, New York, 1996.

\bibitem{MO}
{\sc S.~Mondal, A.~Majumdar, and I.~M. Griffiths}, {\em Nematohydrodynamics for
  colloidal self-assembly and transport phenomena}, Journal of colloid and
  interface science, 528 (2018), pp.~431--442.

\bibitem{MU}
{\sc I.~Musevic, M.~Skarabot, U.~Tkalec, M.~Ravnik, and S.~Zumer}, {\em
  Two-dimensional nematic colloidal crystals self-assembled by topological
  defects}, Science, 313 (2006), pp.~954--958.

\bibitem{MR2837493}
{\sc M.~Paicu and A.~Zarnescu}, {\em Global existence and regularity for the
  full coupled {N}avier-{S}tokes and {$Q$}-tensor system}, SIAM J. Math. Anal.,
  43 (2011), pp.~2009--2049.

\bibitem{MR2864407}
\leavevmode\vrule height 2pt depth -1.6pt width 23pt, {\em Energy dissipation
  and regularity for a coupled {N}avier-{S}tokes and {$Q$}-tensor system},
  Arch. Ration. Mech. Anal., 203 (2012), pp.~45--67.

\bibitem{MR4238201}
{\sc A.~Roy and T.~Takahashi}, {\em Stabilization of a rigid body moving in a
  compressible viscous fluid}, J. Evol. Equ., 21 (2021), pp.~167--200.

\bibitem{MR1870954}
{\sc J.~A. San~Mart\'{\i}n, V.~Starovoitov, and M.~Tucsnak}, {\em Global weak
  solutions for the two-dimensional motion of several rigid bodies in an
  incompressible viscous fluid}, Arch. Ration. Mech. Anal., 161 (2002),
  pp.~113--147.

\bibitem{SEN}
{\sc A.~Sengupta, C.~Pieper, J.~Enderlein, C.~Bahr, and S.~Herminghaus}, {\em
  Flow of a nematogen past a cylindrical micro-pillar}, Soft Matter, 9 (2013),
  pp.~1937--1946.

\bibitem{SER3}
{\sc D.~Serre}, {\em Chute libre d'un solide dans un fluide visqueux
  incompressible. {E}xistence}, Jap. J. Appl. Math., 4 (1987), pp.~99--110.

\bibitem{simon1986compact}
{\sc J.~Simon}, {\em Compact sets in the space {$L^{p} (O, T; B)$}}, Annali di
  Matematica pura ed applicata, 146 (1986), pp.~65--96.

\bibitem{MR3397328}
{\sc T.~Takahashi, M.~Tucsnak, and G.~Weiss}, {\em Stabilization of a
  fluid-rigid body system}, J. Differential Equations, 259 (2015),
  pp.~6459--6493.

\bibitem{wilkinson2015strictly}
{\sc M.~Wilkinson}, {\em Strictly physical global weak solutions of a
  navier--stokes q-tensor system with singular potential}, Archive for Rational
  Mechanics and Analysis, 218 (2015), pp.~487--526.

\end{thebibliography}
\bibliographystyle{siam}

\end{document}